\documentclass[10pt]{amsart}
\usepackage{amsmath,amscd,latexsym,verbatim,amssymb}
\usepackage{times} 
\usepackage[all]{xy}

 \newcommand{\resp}{{\it resp.} }
\newcommand{\cf}{{\it cf.} }
\newcommand{\ie}{{\it i.e.} }
\newcommand{\eg}{{\it e.g.} }

\newcommand{\loccit}{{\it loc. cit.} }

\newcommand{\N}{\mathbb{N}}
\newcommand{\Q}{\mathbb{Q}}
\newcommand{\R}{\mathbb{R}}
\newcommand{\C}{\mathbb{C}}
  \newcommand{\Z}{\mathbb{Z}}

\renewcommand{\P}{\mathbb{P}}

  \newcommand{\sA}{{\mathcal{A}}}
  \newcommand{\J}{{\mathcal J}}

\newcommand{\sH}{{\mathcal{H}}}

\newcommand{\sF}{{\mathcal{F}}}
\newcommand{\sM}{{\mathcal{M}}}

\newcommand{\sO}{{\mathcal{O}}}

\newcommand{\sV}{{\mathcal{V}}}

\newcommand{\sX}{{\mathcal{X}}}

 \newcounter{spec}

\swapnumbers

\newtheorem{thm}{Theorem}[subsection]

\newtheorem{lemma}[thm]{Lemma}
\newtheorem{prop}[thm]{Proposition}
\newtheorem{cor}[thm]{Corollary}
\newtheorem{conj}[thm]{Conjecture}

\theoremstyle{definition}
\newtheorem{defn}[thm]{Definition}

\newtheorem{rem}[thm]{Remark}
\newtheorem{rems}[thm]{Remarks}

\numberwithin{equation}{section}

 at10pt
 at12pt

\renewcommand{\qed}{\hfill $\square$\medskip}
  
 \usepackage{color}

\begin{document}
\begin{sloppypar}

\title[The Betti map associated to a section of an abelian scheme.]{The Betti map associated to a section of an abelian scheme}
 
\author{{Y. Andr\'e, P. Corvaja, U. Zannier}  
 \\
 \\({with an appendix by Z. Gao})}

 \address{Yves Andr\'e\\ Institut Math\'ematique de Jussieu - Paris Rive Gauche\\ 4, Place Jussieu 75005 Paris\\ France }
 \email{yves.andre@imj-prg.fr}
 
 \address{Pietro Corvaja\\ Dipartimento di Scienze Matematiche, Informatiche e Fisiche\\
 Universit\`a di Udine\\ Via delle Scienze, 206\\ Udine\\ Italy}
 \email{pietro.corvaja@uniud.it}
 
 \address{Ziyang Gao\\ Institut Math\'ematique de Jussieu - Paris Rive gauche\\
 4, Place Jussieu 75005 Paris\\ France}
 \email{ziyang.gao@imj-prg.fr}
 
\address{Umberto Zannier\\
Scuola Normale Superiore\\
Piazza dei Cavalieri, 7\\
56126 Pisa\\
Italy}
\email{u.zannier@sns.it}
  
\email{ }
  \date{\today}
  \begin{abstract} Given a point $\xi$ on a complex abelian variety $A$, its abelian logarithm can be expressed as a linear combination of the periods of $A$ with \emph{real} coefficients, the \emph{Betti coordinates} of $\xi$. When $(A, \xi)$ varies in an algebraic family, these coordinates define a system of multivalued real-analytic functions. Computing its rank (in the sense of differential geometry) becomes important when one is interested about how often $\xi$ takes a torsion value (for instance, Manin's theorem of the kernel implies that this coordinate system is constant in a family without fixed part only when $\xi$ is a torsion section). 
 
 We compute this rank in terms of the rank of a certain contracted form of the Kodaira-Spencer map associated to $(A, \xi)$ (assuming $A$ without fixed part, and $\Z \xi$ Zariski-dense in $A$), and deduce some explicit lower bounds in special situations. For instance, we determine this rank in relative dimension $\leq 3$, and study in detail the case of jacobians of families of hyperelliptic curves. 
 
Our main application, obtained in collaboration with Z. Gao, states that if $A\to S$ is a principally polarized abelian scheme of relative dimension $g$ which has no non-trivial endomorphism (on any finite covering), and if the image of $S$ in the moduli space $\sA_g$ has dimension at least $g$, then the Betti map of any non-torsion section $\xi$ is generically a submersion, so that $\xi^{-1}A_{tors}$ is dense in $S(\C)$. \end{abstract}
\maketitle

 \tableofcontents

     \newpage
  
 \section{Introduction I: motivation.}  
 \subsection{}\label{pb} 
 
Let $A\to S$ be an abelian scheme of relative dimension $g$ over a (non necessarily complete) complex algebraic variety, and let $\xi:S\to A$ be a section. We assume that $\xi $ is not identically torsion, \ie does not take value in the kernel $A[n]$ of multiplication by any positive integer $n$ in $A$.
A general problem which arises in a number of contexts concerns the distribution of {\it torsion values} of $\xi$, or more accurately, of the points $s\in S$ such that $\xi(s)$ is a torsion point on the abelian variety $A_s$.

Let us thus introduce the closed subschemes of $S$: 
\begin{equation*}
  \xi^{-1}A[n] = \{s\in S\, :\, n\xi(s)=0\}, \qquad n=1,2,\ldots  
\end{equation*}
 Since $A[n]$ is finite etale over $S$, it is the {\it disjoint} union of {\it closed} subschemes corresponding to exact $m$-torsion for $m\mid n$, and $  \xi^{-1}A[n] $ decomposes accordingly. Each non-empty component 
 has codimension $\leq g$ in $S$, in virtue of Serre's codimension theorem \cite[Th. 3, p. V-18]{S}.

 \medskip
 The following heuristic dichotomy arises: 
 
\noindent (A) When $\dim S <g$, one expects under natural assumptions
   that  $\displaystyle\xi^{-1}A_{\rm{tors}}   $ is contained in  $ \bigcup_{n=1}^N   \xi^{-1}A[n]$ for some $N$, hence is not Zariski-dense. This is a typical question in the theory of  ``unlikely intersections", falling into the framework of the Pink-Zilber conjecture (see \cite{CMZ}).
 
\smallskip\noindent (B) When $\dim S\geq g$, one expects on the contrary, under natural assumptions to be discussed in the sequel, that  $ \xi^{-1}A_{\rm{tors}}$ is dense, even in the complex topology.
This will be one of the main concerns in this work.

 \subsection{} A useful tool to investigate the distribution of torsion values is provided by the so-called {\it Betti map}.   
 Any complex abelian variety $A$ of dimension $g$ may be presented as a complex torus $\C^{g}/\mathcal L\cong (\mathcal L\otimes_\Z \R)/\mathcal L$, and then every point $\xi$ of $A$ may then be identified by its $2g $ real coordinates in a mesh of the lattice $\mathcal L$. The aim of this paper is to understand how these so-called ``Betti coordinates" vary when $(A, \xi)$ moves in an algebraic family: one gets in this way a {\it multivalued real-analytic map} $\beta$, the {\it Betti map}, from the parameter space $S$ to $\R^{2g}$. 
 
   While this problem could have been raised in the XIXth century,  we are not aware of any trace 
 until recent occurrences in algebraic and diophantine geometry, where the relevance of the Betti map arises from the fact that its rational values correspond to torsion values of the section $\xi$. 
 
 \subsection{}  Here is a sample of classical or recent occurrences of torsion value problems and/or Betti maps.

(i) What may be the first explicit connection between torsion value problems and Betti maps can be found in a celebrated paper by Yu. Manin \cite{Man}. He connected the map with differential operators, leading to what is nowadays, after Grothendieck, called the Gauss-Manin connection (for families of curves). A byproduct of his analysis was a characterization of the cases when the Betti map is locally constant. If the abelian family has no fixed part, he proved that {\it the Betti coordinates are locally constant if and only if these constants are rational, \ie if and only if the section is identically torsion}\footnote{for the convoluted story of the proof of this theorem, see for instance \cite{B}\cite[\S 1]{A2}.}. 
\smallskip

(ii) Torsion values occur in connection to {\it Poncelet's game} and its higher dimensional generalizations in a paper by Ph. Griffiths and J. Harris \cite{GH}: given two conics, construct a polygon  which is inscribed in one conic and circumscribed to the other one.  Jacobi proved that the pair of conics determines an elliptic curve with a given point, such that such a polygon can be constructed if and only if the given point is torsion.
Following Jacobi, Griffiths and Harris took a ``variational approach" to this game, by varying the two conics, so obtaining a section of an elliptic family where the base parametrizes the pairs of conics; again torsion comes into play, being connected to the existence of such a polygon (and to the finiteness of the game). They also considered a three-dimensional situation, where conics are replaced by quadric surfaces and polygons by polyhedra. Again, the Betti map appears linking the behaviour of the game with variations of the quadrics.
In this context, the density of the torsion values in the base is far from obvious, and they prove it only in certain hypersurfaces of the base. More recently, the issue of torsion in the context of Poncelet's theorem has been revisited by N. Hitchin, and turned out to be related to the Painlev\'e VI equation \cite{Hi}.
\smallskip

(iii) Torsion values and the Betti map also occur  in the context of {\it Pell equations} in polynomials, which were already studied by Abel in connection with integration of differentials; motivated by this context, B. Mazur recently raised an issue concerning the ``Pellian'' hyperelliptic curves in the space of all hyperelliptic curves of given genus (see \cite{CMZ}). \smallskip 

(iv) These Pell equations in polynomials have number theoretic significance, but also appear in investigations related to the Schr\"odinger equation (see \cite{Krich}). \smallskip

(v) The Pell equation involving the hyperelliptic family of given genus, but now restricting to field of real numbers, appears also in an issue raised by J.-P. Serre, in connection with an incomplete argument in a classical paper by R. Robinson \cite{Rob}. 
An argument for filling in this gap was given by B. Lawrence, and we offer an independent one in Appendix I. 
\smallskip

(vi) Still from a different perspective, C. Voisin \cite{V} considered very recently a closely related problem, this time motivated by the investigation of Chow groups. In the context of Lagrangian fibrations on hyperk\"ahler manifolds,  she uses methods of her own for the case $g\leq 2$, and our Corollary \ref{cor2} to prove the desired conclusion for $g\leq 4$.

  \smallskip

(vii) Betti coordinates also occur in a forthcoming work by Z. Gao and Ph. Habegger on the geometric Bogomolov conjecture \cite{GaoHabegger}.

 \subsection{}  In this paper, we undertake a systematic study of the Betti map $\beta$ associated to $(A, \xi)$, and prove the density of torsion values under some fairly natural hypotheses.  
    We relate the derivative of the Betti map to the Kodaira-Spencer map.
 Functional transcendence also appears, through a monodromy theorem by the first author, but also via a very recent result, namely the proof of a version of Ax-Schanuel conjecture in the context of subvarieties of the moduli space $\sA_g$ and its uniformization.

  \section{Introduction II: main results.}  
      
      \subsection{Rank of the Betti map and torsion values} Let $A\to S$ be an abelian scheme of relative dimension $g$ over a  smooth complex algebraic variety of dimension $d$, and let $\xi: S\to A$ be a section. Let $\beta: \tilde S \to \R^{2g}$ be the associated (real analytic) Betti map, where $\tilde S$ denotes the universal covering of $S(\C)$ (\cf \S 3 for a precise definition and discussion of this map)\footnote{actually, it also depends on a choice of branch of abelian logarithm $\lambda$, \cf \ref{bet}.}. In this paper, we focus on the {\it (generic) rank of $\beta$} as defined in differential topology, namely, the maximal value of the rank of the derivative $d \beta (\tilde s)$ when $\tilde s $ runs through $\tilde S$. We denote it by $\,{\rm{rk}} \, \beta$.

 As explained above, the motivation for studying ${\rm{rk}} \, \beta$ comes from the relation to torsion values of $\xi$. The following proposition, which is a simple consequence of the constant rank theorem and of the fact that $\Z[1/p]$ is dense in $\R$, makes this relation explicit.

 \begin{prop} The following are equivalent:
 \begin{enumerate} \item ${\rm{rk}} \, \beta\geq 2g$,
 \item $\beta$ is a submersion on a dense open subset of $\tilde S$,
 \item the image of $\beta$ contains a dense open subset of $\R^{2g}$,
 
 and imply\footnote{we don't know whether the converse holds: the image of $\beta$ might be dense without containing a dense open subset.}:
 
 \item for every prime $p$, $\,\xi^{-1}A[p^\infty]$ is dense in $S(\C)$ (for the complex topology),
 \item $  \xi^{-1}A_{tors}$ is dense in $S(\C)$.   \qed
  \end{enumerate}  
  \end{prop}

Because  ${\rm{rk}} \, \beta$ is invariant by dominant pull-back $S'\to S$, one may assume that for fixed $n\geq 3$, $A[n]$ is a disjoint union of copies of $S$ (\ie $n$-torsion is $S$-rational). One may also replace $A$ by any isogenous abelian scheme and $\xi$ by the corresponding pull-back,  and assume in particular  that $A$ is principally polarized. 

This gives then rise to a modular map $\mu_A: \, S \to \sA_{g,n}$ such that $A$ is the pull-back by $\mu_A$ of the universal principally polarized abelian scheme (with level $n$ structure) $\sX_{g,n}$ on $\sA_{g,n}$. 
 Moreover $\xi$ gives rise to a modular map $\mu_\xi: \, S \to \sX_{g,n}$ lifting $\mu_A$
    \[  \xymatrix@-1pc {A   \ar[d] \ar[r] &  \sX_{g,n} \ar[d]  \\  S  \ar[ru]^{\mu_\xi}     \ar[r]_{\mu_A} &  \sA_{g,n}} \] 
 such that $\xi$ is the pull-back by $\mu_\xi$ of the diagonal map  $ \sX_{g,n}\to  \sX_{g,n}\times_{\sA_{g,n}}\sX_{g,n}$.

       If $d_{\mu_\xi}$ denotes the dimension of the image of $\mu_\xi$, we have the upper bound
      \begin{equation} \label{E.1.8}
   \,{\rm{rk}} \, \beta\, \leq 2\cdot\min ({d_{\mu_\xi}}, g).
      \end{equation}

     Of course, this inequality may be strict in some ``degenerate" cases, \eg when $\xi$ is the $0$ section or when $(A, \xi)$ is constant, but  
   we conjecture: 
 
 \begin{conj} If $\,\Z \xi\,$ is Zariski-dense in $A$, and $A$ has no fixed part (over any finite etale covering of $S$), then \eqref{E.1.8} is an equality. 
 \end{conj}

 \subsection{Formulas for ${\rm{rk}} \, \beta$.}  
  Passing to universal coverings, $\mu_A$ lifts to a map $$\tilde \mu_A := \tilde S \stackrel{\tilde s \mapsto Z(\tilde s)}{\to} \frak{H}_g$$ toward the Siegel space, where  $Z = \Omega_2 \Omega_1^{-1}$ is given in terms of abelian periods as usual. Let also $L:\tilde{S}\to \C^g$ stand for an abelian logarithm  of the section $\xi:S\to A$ multiplied on the right by $\Omega_1^{-1}$ (\cf \ref{sn}). 

Our first formula is analytic:

\begin{thm}\label{rk1} Assume that $A$ has no fixed part (over any finite etale covering of $S$), and $\Z \xi$ is Zariski-dense in $A$. Up to replacing $S$ by an affine dense open subset, we may assume that its tangent bundle is free, generated by $d$ derivations $\partial_1, \ldots, \partial_d$. Then \begin{equation}  {\rm{rk}} \, \beta = 2 \cdot \max_{\mu_0, \ldots, \mu_g\in \C^{g+1}} \max_{\tilde s\in \tilde S}\,{\rm{rk}}  \begin{pmatrix} \mu_0 \partial_i L_j(\tilde s) - \sum_1^g \mu_k \partial_i Z_{kj}(\tilde s) \end{pmatrix}_{ij}. \end{equation} 
\end{thm}   
 
 By letting $\mu_0=0$, one deduces the following criterion for generic submersivity, which does no longer involve the section:
 
 \begin{cor}\label{cor1} In the same situation, assume that $d\geq g\,$  and that  $ {\rm{rk}} \, \beta < 2g$. Then for any $\tilde s\in \tilde S$ and any $\mu \in \C^g$, there exists a complex-analytic subvariety of $\tilde S$ passing through $\tilde s$, of dimension $\geq d+1-g$, on whicht $\,Z.\mu\,$ is constant. 
 \end{cor}

Our second formula is nothing but an algebraic intrinsic reformulation of the former in terms of Kodaira-Spencer maps (\cf \S \ref{fhtr}).  
Recall the
 Kodaira-Spencer map of $A$ is the $\mathcal{O}_S$-linear map
     \begin{equation} \theta_A: \, T_S \otimes  \Omega_A \to  \mathcal{H}^1_{\mathrm{dR}}(A/S)/\Omega_A \cong Lie \,A,  \end{equation}
    induced by the Gauss-Manin connection. Note that the vector bundle $ \Omega_A$ of invariant $1$-forms is the $gr^1$-notch of the Hodge filtration, whereas its dual $Lie \,A$ appears as the   $gr^0$-notch via the polarization  (see \S \ref{KS} for a presentation and discussion of $\theta_A$).

The Kodaira-Spencer map of the $1$-motive $[\Z \stackrel{1\mapsto \xi}{\to} A  ]$ attached to $(A,\xi)$ (in the sense of \cite[X]{D1}) is an enhanced version of $\theta_A$
   \begin{equation} \theta_\xi: \, T_S \otimes  gr^1  \to gr^0,  \end{equation} 
   where $gr^1 = \Omega_A$ again, but $gr^0$ is an extension of  $Lie\, A$ by $\sO_S$ (\cf \ref{KS}).
 
 Assuming for simplicity that $S$ is {\it affine}, we denote by  \begin{equation} \theta_\xi^\varpi: T_S   \to Lie\,A = (gr^1)^\vee \end{equation}  the contracted map, for any section $\varpi$ of $ (gr^0)^\vee$.
                
      \begin{thm}\label{rk2} Assume that $A$ has no fixed part, and $\Z \xi$ is Zariski-dense in $A$. Then 
      \begin{equation} 
      \displaystyle\,{\rm{rk}} \,  \beta \,  = 2\cdot   \max_{(\varpi_0, \ldots, \varpi_g)\in (gr^0)^\vee} \max_{s\in S} \,{\rm{rk}} \,\theta_\xi^\varpi(s). 
      \end{equation}   
      \end{thm}    
    
    By letting $\varpi_0=0$, one deduces the following criterion for generic submersivity in terms of $\theta_A$ (which again does not involve the section)\footnote{this is the criterion used in \cite[th. 0.6]{V}.}:

     \begin{cor}\label{cor2} Under the same assumptions, if there exists $\omega\in \Gamma \Omega_A$ such that the  map  \begin{equation} \theta^\omega_A: \,T_S \,  {\to} \,Lie\,A, \;\:\partial \mapsto \theta_A(\partial\otimes \omega)  \end{equation}   has generic rank $g$, then $\,{\rm{rk}} \,  \beta \geq 2g$. 
     \end{cor}
     
  \smallskip Let us give an indication on the method of proof of Theorem \ref{rk1}.  For any $\nu\in \C^{2g}$, let $I_\nu$ be the $d\times g$-matrix  with entries   $\displaystyle    \partial_i \Lambda_j   +  \sum_k \nu_k   \partial_i\Omega_{kj} $, where $\Lambda$ is a vector of abelian logarithms and $\Omega_{jk}$ a matrix of abelian periods (\cf \eqref{e5} below). The rank of $d\beta(\tilde s)$  can be computed in terms of vanishing/non vanishing of minors of $I_{-\beta(\tilde s)}$ (\cf \eqref{e7}). The point is then to show that the vanishing of such a minor identically on $\tilde S$ implies the vanishing of the corresponding minor of $I_\nu$ for any $\nu$. 
        
        The idea is to let monodromy act and use a theorem about relative monodromy of abelian logarithms with respect to abelian periods \cite[Th. 3]{A1}. This does not work directly because $I_{-\beta}$ is monodromy-invariant. The trick is that, after multiplication by some power of $\det \begin{pmatrix} \Omega_1 & \bar\Omega_1\\  \Omega_2 & \bar\Omega_2\end{pmatrix}$,  $I_{-\beta}$ becomes a polynomial in holomorphic and antiholomorphic multivalued functions on $S$, so that one can let two copies of $\pi_1(S)$ act on $\sO(\tilde S)\otimes \bar\sO(\tilde S)$, instead of standard monodromy (which corresponds to the diagonal action).

   \subsection{Applications} Let $d_{\mu_A}\leq d $ be the dimension of the image of the modular map $\mu_A$.
   
\noindent Using Corollary \ref{cor2}, an elementary analysis of webs of quadrics, combined with a result of \cite{A3}, allows to settle the case $g\leq 3$:

\begin{thm}\label{tg=3}
Suppose that the abelian scheme $A\to S$ has relative dimension $g\leq 3$, has no fixed part, and $d_{\mu_A}\geq g$. Then for every section $\xi$ not contained in a proper subgroup scheme, $\,{\rm{rk}} \,  \beta \geq 2g$.
\end{thm}

  Our {\it main application}, in collaboration with Z. Gao, is:
          
             \begin{thm}\label{lbrk}({\rm{with Z. Gao}}) Assume that the abelian scheme $A\to S$ has no non-trivial endomorphism over any finite covering of $S$, and that the section $\xi $ is non-torsion. Assume moreover that $d_{\mu_A}\geq g$. Then $\,{\rm{rk}} \,  \beta \geq 2g$, hence the set of points $s$ such that $\xi(s)$ is a torsion point is dense in $S(\C)$.  
      \end{thm} 
         
         This is obtained by combining Corollary \ref{cor1}, a classification of the abelian schemes under consideration (\S 8), and the recent theorem of Ax-Schanuel for $\sA_g$ by Mok, Pila, Tsimerman (cf. App. 2).\footnote{In a previous version, an explicit  link of these issues with the  "Ax-Schanuel conjecture" was pointed out.
Very recently this conjecture has been proved, making thus possible to apply it to our problem, as done by Gao in the  appendix.}

 \subsection{} We also settle the {\it hyperelliptic case}:
       let $A\to S$ be the jacobian of the universal hyperelliptic curve of genus $g>0$. Here $S $ is the affine space $\sM_{0, 2g+2}\cong (\mathbb P^1\setminus \{ 0,1,\infty \})^{2g-1}$. By Torelli's theorem, one has  $\dim \, \mu_A(S) = 2g-1$. 
       
       \begin{thm}\label{hyp}  Let   $S$ be a finite cover    of 
       $ (\mathbb P^1\setminus \{0,1, \infty\})^{2g-1}$ and $\xi: S \to A$ be any non-torsion section. Then  $\,{\rm{rk}} \,  \beta \geq 2g$.  \end{thm} 

While this may be considered as a consequence of the previous theorem \cite{Mo}\cite{aC}, we offer an elementary proof based on Corollary \ref{cor2} and an explicit computation of the Kodaira-Spencer map.

\smallskip As mentioned above, a real version of the hyperelliptic case is of particular interest and is treated in Appendix I (independently of the main body of the paper); it is connected with an issue recently raised  by Serre, who pointed out a missing justification in an  old work of Robinson on transfinite diameters \cite{Rob}. 
We realized that this issue, although concerning only real points, fits into the present context and can be treated considering Betti maps. 
 For a point $s=(s_1,\ldots,s_{2g})\in S=\mathbb P^1\setminus\{0,1,\infty\})^{2g} $ with pairwise distinct coordinates, consider the corresponding genus $g$ hyperelliptic curve
\begin{equation*}
    y^2=x(x-1)(x-s_1)\cdots (x-s_{2g}).
\end{equation*}
 In a smooth compactification, it has two points at infinity, denoted as $\infty^+,\infty^-$. Denoting again by $A\to S$ the jacobian scheme of the curve defined by the above equation, let us consider the section $\xi:S\to A$ associating to each $s\in S$ the class of the divisor $\infty^+ - \infty^-$. 
Both the abelian scheme $A\to S$ and the section $\xi:S\to A$ are defined   over the reals, so in particular for each real point $s\in S(\R)$ the value $\xi(s)$ is a real point of the abelian variety $A_s$ (which is defined over $\R$). The tangent space $Lie A$ also inherits a real structure, and we can define the real version of the Betti map as a map $\beta_\R: \tilde{S}(\R)\to \R^g$.  We prove that $\beta_\R$ is generically a surjection, thus filling the gaps related to Serre's question.

   \section{The Betti map.}
   
      \subsection{Betti coordinates on an abelian variety}\label{bet} Let us first clarify the ``constant case", \ie when $S$ is a point. Given a complex abelian variety $A$, we identify its {\it period lattice} $\mathcal L := \,H_1(A(\C), \Z)$ with the kernel of the exponential map   \begin{equation} \exp_A: Lie\, A \to A.  \end{equation} 
    Given $\xi\in A(\C)$, we denote by $\mathcal L_\xi$ the lattice in $(\mathcal L\otimes_\Z \R) \oplus\R$ consisting of pairs $(\ell \in Lie\, A, m \in \Z) $ such that $exp_A (\ell) = m\xi\,$; it sits in an extension 
   \begin{equation} 
   0\to \mathcal L \to \mathcal L_\xi \stackrel{(\ell, m)\mapsto m\xi}{\to} \Z\xi \to 0.
   \end{equation}
   We fix a splitting
    $\lambda: \Z\xi \to \mathcal L_\xi\,$ (\ie a branch of the ``abelian logarithm"). 
   On the other hand, let us consider the first projection $\,\mathcal L_\xi \stackrel{(\ell, m)\mapsto \ell}{\to}  Lie\, A$ and its $\R$-linear extension $\mathcal L_\xi\otimes_\Z \R \to Lie\, A\,$. Since the composed map 
    \begin{equation}\label{expo} \mathcal L\otimes_\Z \R \to  \mathcal L_\xi\otimes_\Z \R \to Lie\, A \end{equation} is an isomorphism, we get a retraction  
       \begin{equation}
{\mathcal L_\xi\otimes_\Z \R \to   \mathcal L\otimes_\Z \R},    
       \end{equation}
        and the image of $\lambda(\xi)\otimes 1$ in $ \mathcal L\otimes_\Z \R$ can be expressed, in terms of a basis $\underline \gamma$ of $\mathcal L$, by $2g$ real coordinates, the {\it Betti coordinates}
      of $\xi\,$\footnote{This terminology, due to D. Bertrand, refers to the fact that these are indeed real-analytic {\it coordinates} on any simply-connected domain in $A(\C)$, and that $\mathcal L_\xi$ is nothing but the {\it Betti realization of the $1$-motive} $[\Z\stackrel{1\mapsto \xi}{\to} A]$ attached to $(A, \xi)$ in the sense of \cite[X]{D1}; it is however slightly abusive, since these coordinates also depend on the choice of a basis $(\lambda(\xi), \underline \gamma)$ of $\mathcal L_\xi$.}. 
    
    In terms of abelian integrals, the Betti coordinates of $\xi$ are nothing but the $2g$ {\it real} solutions $\beta_i$ of the system of $g$ inhomogeneous linear equations with {\it complex} coefficients   \begin{equation}\label{int} \displaystyle  \int^\xi \omega_j  = \sum_{i=1}^{2g} \beta_i  \int_{\gamma_{i }} \omega_j, \;\, j= 1, \ldots g\, . \end{equation}

  \subsection{The Betti map attached to a section of an abelian scheme} As explained above, we are interested in the relative setting, that is, in the {\it variation} of the Betti coordinates in the context of a family of abelian varieties endowed with a section. 
  
  Let $S$ be a smooth connected complex algebraic variety, and let $A\stackrel{f}{\to} S$ be an abelian scheme of relative dimension $g$. Its Lie algebra $Lie\, A$ is a rank $g$ vector bundle on $S$.

Let $\xi: S\to A$ be a section of $f$. 
The above constructions extend as follows. Let $\tilde S$ be a universal covering of $S(\C)$, with its canonical structure of complex analytic manifold. 
The kernel of
$ \exp_A$ is a locally constant sheaf on $S(\C)$, which can be viewed as a constant lattice $\mathcal L$ on $\tilde S$.
Similarly, one constructs an exact sequence of lattices on $\tilde S$
    \begin{equation} 0\to \mathcal L \to \mathcal L_\xi \stackrel{(\ell, m)\mapsto m}{\to} \Z\xi \to 0 \end{equation} and one fixes a splitting $\lambda: \Z\xi \to \mathcal L_\xi$.

      Let $\tilde S^{real}$ be the real-analytic manifold underlying $\tilde S$, and let $\sO^{\R}_{{\tilde S}^{real}}$  (\resp $\sO_{{\tilde S}^{real}}$)  be the sheaf of real-valued (\resp complex valued) real-analytic functions on ${\tilde S}$.
  The relative version of \eqref{expo} states that the composed map  
   \begin{equation} \mathcal L\otimes_\Z\sO^{\R}_{{\tilde S}^{real}}\to  \mathcal L_\xi\otimes_\Z\sO^{\R}_{{\tilde S}^{real}} \to   Lie\, A\otimes_{ O_S} \sO_{{\tilde S}^{real}} 
   \end{equation} 
is an isomorphism, whence a canonical retraction  $\mathcal L_\xi\otimes_\Z \sO^{\R}_{{\tilde S}^{real}} \to   \mathcal L\otimes_\Z \sO^{\R}_{{\tilde S}^{real}}$, and the image of $\lambda(\xi)\otimes 1$ as a global section of $ \mathcal L\otimes_\Z \sO^{\R}_{{\tilde S}^{real}}$ 
   can be expressed, in terms of a basis $\underline \gamma$ of the lattice $\mathcal L$, by a real-analytic map 
  \begin{equation}\beta = \beta_{\lambda, \underline \gamma} : \;\; \tilde S^{real}  \to \R^{2g}, 
  \end{equation}  
  which we call the {\it Betti map} associated to $(A,\xi)$, or to the smooth 1-motive $[\Z \stackrel{1\mapsto \xi}{\to} A]\,$ (the terminology is slightly abusive, since $\beta$ also depends on the choice of a basis $(\lambda(\xi), \underline \gamma)$ of $\mathcal L_\xi$). It is compatible with pull-back of $(A, \xi)$ by any morphism $S'\to S$, with $S'$ smooth connected.

     The interpretation of $\beta$ in terms of abelian integrals is the same as in \eqref{int} (``real" and ``complex" being understood as ``real-valued" and ``complex-valued" respectively), and shows that the fibers of $\beta$ are complex-analytic subvarieties of $\tilde S$ (\cf \ref{beg}). 

     The non-empty pre-images of the Betti map are complex subvarieties of dimension  $\geq \dim S-\frac{1}{2}{\rm{rk}} \, \beta\,$ (\cf \cite[prop. 2.2]{CMZ}).

  \section{From Betti to Kodaira-Spencer.  }
  
  In this section, we relate the derivative of the Betti map $\beta$ at an arbitrary point $\tilde s \in \tilde S$ to the corresponding value $\theta_\xi^\varpi(s)$ of the contracted Kodaira-Spencer map for some {\it specific} parameter $\varpi$ (depending on $\tilde s$). Although we work, as above, over a (smooth) complex algebraic variety $S$ for convenience, the metamorphoses of the derivative of the Betti map which we display do not actually require the algebraicity of $S$.  
  
 \subsection{Setting and notation}\label{sn} Let $A\to S$ be an abelian scheme of relative dimension $g$ over a smooth connected algebraic $\C$-variety of dimension $d$. For any holomorphic function $g$ on the universal covering $\tilde S$ of $S$, we denote by $\bar g$ the conjugate antiholomorphic function. 
 
    The problems studied in this paper are Zariski-local on $S$, and insensitive to  replacing $A$ by an isogenous abelian scheme. Therefore we may and shall assume henceforth that 
 
\smallskip $i)$ $S$ is affine and admits a system of global (etale) coordinates $(z_1, \ldots, z_d)$; we take $(\partial_1= \frac{\partial}{\partial z_1}, \ldots, \partial_d= \frac{\partial}{\partial z_d})$ as a basis of tangent vector fields to $S$, also viewed as a basis of holomorphic derivations of $\sO(\tilde S)$,
 
 $ii)$ the vector bundle $Lie\,A$ is trivial, \ie  $\Gamma Lie\, A $ is a free $\Gamma \sO_S$-module of rank $g\,$; its dual is canonically isomorphic to the sheaf $\varOmega_A$ of invariant $1$-forms (a.k.a. differentials of the first kind),
 
 $iii)$ the vector bundle  $\sH^1_{dR} (A/S)$ is trivial,
 
 $iv)$ $A$ is principally polarized.

\smallskip We fix a basis $ (\omega_1,\ldots , \omega_g)$ of global sections of $\varOmega_A\cong  Lie\, A $, and complete it into a symplectic basis $(\omega_1,\ldots , \omega_g, \eta_1,\ldots , \eta_g)$ of $\sH^1_{dR} (A/S)$ (which carries the symplectic form coming from the polarization).

We fix a symplectic basis $(\gamma_1,\ldots , \gamma_{2g})$ of $\mathcal L$.

  We set  $\,\Omega_1 := \begin{pmatrix} \int_{\gamma_i} \omega_j\end{pmatrix}_{i, j= 1,\ldots g}, \; \Omega_2 := \begin{pmatrix} \int_{\gamma_{i+g}} \omega_j\end{pmatrix}_{i, j= 1,\ldots g},  \Omega := \begin{pmatrix} \Omega_1\\ \Omega_2\end{pmatrix}, \;$
    
    The entries of  $\Omega$ are holomorphic on $\tilde S$, and the  $g\times g$-matrices $\Omega_1, \Omega_2$ and the $2g\times 2g$-matrix $\begin{pmatrix} \Omega , \overline\Omega \end{pmatrix}$ are invertible at every point ${\tilde s}\in \tilde S$. 
     
 \smallskip Finally, we fix a section $\xi$ of $A/S$ and a determination $\lambda$ of the abelian logarithm as in \ref{bet}. We may and shall assume that 
 
 $v)$ the extension $\sH^1_{dR}([\Z \stackrel{1\mapsto \xi}{\to} A]/S)$ of $\sH^1_{dR} (A/S)$ by $\sO_S$ is trivial.

We set   $ \Lambda :=  \begin{pmatrix} \int^\xi \omega_j \end{pmatrix}_{j= 1,\ldots g}   $, a row of $g$ holomorphic functions on $\tilde S$. 
       
  \subsection{From real-analytic to holomorphic maps}\label{beg}  We write the Betti map $\beta= \beta_{\lambda, \underline\gamma}$ of the 1-motive $[\Z \stackrel{1\mapsto \xi}{\to} A]$ (in the basis $(\lambda(\xi), \underline \gamma)$ of $\mathcal L_\xi$) as a {\it row} with $2g$ entries: by definition, $\lambda(\xi) \equiv \sum_1^{2g} \beta_i \gamma_i$ in $\mathcal L_\R$, so that      
    \begin{equation}\label{e1} \Lambda = \beta \Omega  .\end{equation} 
     In particular, {\it the fibers $\beta^{-1}(b)$ of $\beta$ are the complex-analytic subvarieties of $\tilde S$} defined by $\Lambda({\tilde s}) = b \Omega({\tilde s})    $. 
  
  From \eqref{e1}, one gets
 $\begin{pmatrix} \Lambda , \overline\Lambda \end{pmatrix} = \beta \begin{pmatrix} \Omega , \overline\Omega \end{pmatrix} $ since $\beta$ is real-valued, whence 
     \begin{equation}\label{e2} \beta = \begin{pmatrix} \Lambda , \overline\Lambda \end{pmatrix} \begin{pmatrix} \Omega , \overline\Omega \end{pmatrix}^{-1}.  \end{equation}
    In particular, the entries of
     $$ \det\begin{pmatrix} \Omega , \overline\Omega \end{pmatrix} \cdot \beta$$ are polynomials in holomorphic and antiholomorphic functions on $\tilde S$.
     
     \smallskip
     It will be convenient to build the Jacobian matrix $J = J_{ \lambda, \underline\gamma } $ of the Betti map using the system of derivations $(\partial_1, \ldots, \partial_d, \bar\partial_1=\frac{\partial}{\partial \bar z_1}, \ldots, \bar\partial_d= \frac{\partial}{\partial \bar z_d})$:  
\begin{equation}\label{e3} J = \begin{pmatrix} \partial_1 \beta_1 & \ldots & \partial_1 \beta_{2g}\\ \vdots& \, & \vdots\\  \bar\partial_d \beta_1 & \ldots & \bar\partial_d \beta_{2g} \end{pmatrix}. \end{equation}  

Differentiating \eqref{e2}, one gets 
     \begin{equation}\label{e4} \partial_i\beta \begin{pmatrix} \Omega , \overline\Omega \end{pmatrix} =  \begin{pmatrix} \partial_i \Lambda , 0 \end{pmatrix} - \beta \begin{pmatrix}   \partial_i\Omega , 0 \end{pmatrix} , \;\; \bar\partial_i\beta \begin{pmatrix} \Omega , \overline\Omega \end{pmatrix}=  \begin{pmatrix} 0, \bar\partial_i\overline \Lambda   \end{pmatrix} - \beta \begin{pmatrix} 0,   \bar\partial_i\overline\Omega  \end{pmatrix}.   \end{equation}  
     For any $\nu\in \C^{2g}$, let $I_\nu$ be the $d\times g$-matrix with holomorphic entries    
        \begin{equation}\label{e5}\displaystyle (I_\nu)_{ij} :=  \partial_i \Lambda_j   +  \sum_k \nu_k   \partial_i\Omega_{kj}.      \end{equation} 
Combining equations \eqref{e4} for $i= 1, \ldots, d$, one gets
           \begin{equation}\label{e6}  J({\tilde s}) \cdot \begin{pmatrix}  \Omega , \overline\Omega \end{pmatrix}({\tilde s}) =
         \begin{pmatrix}   I_{-\beta({\tilde s})}  &  0 \\ 
         0& \overline I_{-\beta({\tilde s})}\end{pmatrix}({\tilde s}), \end{equation}
 so that for every $s\in \tilde S$, 
   \begin{equation}\label{e7} {\rm{rk}}\,J({\tilde s})= {\rm{rk}}\, J({\tilde s})\begin{pmatrix} \Omega , \overline\Omega \end{pmatrix}({\tilde s}) = 2\, {\rm{rk}}\,I_{-\beta({\tilde s})}({\tilde s}). \end{equation} In particular, ${\rm{rk}}\,J({\tilde s})$ {\it is even}.
 
\begin{rems}\label{remI} $(1)$ Just like  $ \det\begin{pmatrix} \Omega , \overline\Omega \end{pmatrix} \cdot \beta$, the entries of 
$$ \det\begin{pmatrix} \Omega , \overline\Omega \end{pmatrix}  I_{-\beta} $$ are polynomials in holomorphic and antiholomorphic functions on $\tilde S$. 

\smallskip\noindent $2)$ On the other hand, the entries of $ I_{-\beta} $ are monodromy-invariant (hence real-analytic functions on $S$). Indeed, there is a natural action of the deck transformation group $\Delta := {\rm{Aut}} (\tilde S/S)$ on $\mathcal L_\xi$ which preserves $\mathcal L$, and for any $\delta\in \Delta$,  one can write $\delta \underline \gamma = M_\delta \underline \gamma$ and $\displaystyle \delta\lambda(\xi) =   \lambda(\xi)  + \sum_k \nu_{\delta, k} \gamma_k$ (where the entries of $M_\delta$ and $\nu_\delta$ are integers), whence (by integration of holomorphic differentials along these cycles) 
      \begin{equation}\label{e9'}  \Omega(\delta {\tilde s})  = M_\delta \Omega({\tilde s}), \;\;  \Lambda(\delta {\tilde s})  =  \Lambda({\tilde s})  + \nu_\delta \Omega({\tilde s}).\end{equation}  
  Using \eqref{e2}, one gets
   $$(3.8)' \;\;\;       \beta_{\lambda, \underline \gamma}({\tilde s}) + \nu =   (\beta_{\delta\lambda, \delta\underline \gamma}(\delta {\tilde s})) M ,$$
 whence  
 $$(3.9)' \;\;\;   J_{\lambda, \underline \gamma}({\tilde s})  = J_{\delta\lambda, \delta\underline \gamma}(\delta {\tilde s}) M ,$$ and further 
 $$(3.11)'  \;\;\;    I_{-\beta({\tilde s})}(\delta {\tilde s}) = I_{-\beta({\tilde s}) }({\tilde s}) .$$ 
This shows that {\it $rk\, \beta(\tilde s)$ depends only on the point $s\in S$ under $\tilde s$. In fact the same calculation shows that it does not even depend on the auxiliary choice of $(\lambda, \underline{\gamma})$}.     
 \end{rems}

    \subsection{Going to the Siegel space $ \frak H_g$}\label{sub-sectionSiegel} We set  
  \begin{equation*}
    Z   := \Omega_2\cdot \Omega_1^{-1}, \; \;   L := \Lambda\cdot \Omega_1^{-1} 
    \end{equation*}
  and note that $Z$ takes values in
  $ \frak H_g$, \ie $Z$ is symmetric and $\rm{Im} \, Z >0$.
    
 From $ \Lambda = \beta \Omega $, one gets          \begin{equation}\label{e12} \beta = \begin{pmatrix} L , \overline L \end{pmatrix} \begin{pmatrix}   I& I \\ Z & \overline Z\end{pmatrix}^{-1},  \end{equation} 
    (where $I$ stands for the $g\times g$ identity matrix), whence 
           \begin{equation}\label{e12''}  \beta_1   + \beta_2 Z = L  \end{equation} 
  and  $\beta = (\beta_1, \beta_2)$ with     \begin{equation}\label{e12'}\beta_2 =  \Im L (\Im Z)^{-1}, \;\beta_2 =   \frac{1}{2} (- L(\Im Z)^{-1}\bar Z + \bar L  (\Im Z)^{-1} Z )  \end{equation} (where $\Im$ stands for the imaginary part).

   \smallskip   For any $\mu\in \C^{g}$, let $H_\mu$ be the $d\times g$-matrix with holomorphic entries   
         \begin{equation}\label{e13}\displaystyle (H_\mu)_{ij} := \partial_i L_j  +  \sum_k \mu_k  \partial_iZ_{kj}.     \end{equation} 
     On denoting by ${\beta({\tilde s})_1} $ the first half of the row $ \beta({\tilde s})$,  one draws as above (\cf \eqref{e3} \eqref{e6}):    
     \begin{equation}\label{e14} 
     J({\tilde s}) \cdot  \begin{pmatrix}  I& I\\ Z & \overline Z\end{pmatrix}({\tilde s}) =
         \begin{pmatrix}   H_{-\beta({\tilde s})_2}  &  0 \\ 
         0& \overline H_{-\beta({\tilde s})_2}\end{pmatrix}({\tilde s}). 
   \end{equation}
          
          A straightforward calculation shows that    \begin{equation}\label{e18}\displaystyle  \partial  \Lambda + \nu \partial  \Omega = ( \partial  L + \nu_2  \partial  Z)\Omega_1 + (L +\nu_1   +   \nu_2 Z)\partial  \Omega_1,
       \end{equation}  
     so that
  \begin{equation}\label{e18'} H_{\nu_2}({\tilde s}) \cdot \Omega_1({\tilde s}) = I_\nu({\tilde s})  \end{equation}
   if  $\nu = (\nu_1 =  - L(\tilde s) - \nu_2 Z(\tilde s) ,   \, \nu_2)$, \ie if $\,\nu_1\Omega_1 + \nu_2\Omega_2 = -\Lambda$.
   
    For $\nu = -\beta(\tilde s)$, one recovers \eqref{e14} from \eqref{e6} using \eqref{e12''}; more generally, 
    \begin{equation}\label{e18''}\displaystyle \max_{\nu\in \C^{2g}} {\rm{rk}}\,I_\nu({\tilde s})  \geq   \max_{\mu\in \C^{g}} {\rm{rk}}\, H_\mu({\tilde s}) . \end{equation}

     \subsection{From holomorphic functions to rational functions}\label{fhtr}  
   We set 
     $\; M :=  \begin{pmatrix} \int^\xi \eta_j \end{pmatrix}_{j= 1,\ldots g} . $
     The $(1+2g )\times(1+2g)$-matrix $Y_\xi := \begin{pmatrix} 1&\Lambda & M\\ 0 &\Omega_1&N_1 \\ 0 & \Omega_2 & N_2 \end{pmatrix}$  is a full solution of the Gauss-Manin connection  attached to the family of 1-motives $[\Z \stackrel{1\mapsto \xi}{\to} A]$, which has the form\footnote{as in \cite{A3}, we write this differential system in a slightly non-standard way, with the matrix of the connection on the right so that the monodromy acts on the left.  $\ell_\partial$ is essentially the row of rational functions which occur as second member of the inhomogeneous scalar differential operator in Manin's kernel theorem.}
     \begin{equation}\label{e19} \partial Y_\xi = Y_\xi \begin{pmatrix} 0&\ell_\partial & m_\partial\\ 0 &R_\partial &S_\partial  \\ 0 & T_\partial & U_\partial   \end{pmatrix},  \end{equation} where the entries of the last displayed matrix lie in $\sO(S)$, and depend linearly on the tangent vector field $\partial \in \Gamma T_S$.      
  
 The advantage of going to the Siegel space, \ie of considering $H_\mu $ rather than $I_\nu$, appears in the following calculations based on \eqref{e19}: 
   \begin{equation}\label{e20} (\partial L)\cdot \Omega_1 = \ell_\partial +  (M   - \Lambda N ) T_\partial \end{equation}
 while 
  \begin{equation}\label{e21}  (\partial Z) \cdot \Omega_1  =  \partial \Omega_2 -\Omega_2\Omega_1^{-1}\partial \Omega_1     = 
       (  \Omega_2R_\partial + N_2 T_\partial) - \Omega_2\Omega_1^{-1} (\Omega_1 R_\partial + N_1 T_\partial )    \end{equation}
    $ =
      (N_2 -  \Omega_1\Omega_1^{-1} N_1)T_\partial  =
          {}^t\Omega_1^{-1} ({}^t\Omega_1 N_2 - {}^t\Omega_2 N_1 ) T_\partial   =  2\pi i    {}^t\Omega_1^{-1} T_\partial    $

          \noindent (using the symmetry of $Z$ and, at the end, the analogue of the Legendre relation for abelian varieties).
          
   \smallskip Let $s$ be the image of $\tilde s $ in $S$. For any $\mu\in \C^{g}$, let $G_\mu$ be the $d\times g$-matrix with rational entries  
        \begin{equation}\label{e22}\displaystyle (G_\mu)_{ij} := (\ell_{\partial_i})_j   +  \sum_k \mu_k  (T_{ \partial_i})_{kj}\in \sO(S).     \end{equation}
           
   It follows from \eqref{e20} and \eqref{e21} that
           \begin{equation}\label{e23}\displaystyle  H_\mu  \Omega_2  =   G_{M  - \Lambda N  +2\pi i  \mu\, {}^t\Omega_1^{-1} }  ,\end{equation} 
     and since for fixed $\tilde s$, $\mu \mapsto  M(\tilde s)   - \Lambda({\tilde s})N(\tilde s) +2\pi i  \mu\, {}^t\Omega_1^{-1}({\tilde s})$ is an affine automorphism of $\C^g$, one concludes that
       \begin{equation}\label{e24} \max_{\mu\in \C^{g}} {\rm{rk}}\, H_\mu({\tilde s}) =     \max_{\mu\in \C^{g}} {\rm{rk}}\, G_\mu(s) .\end{equation}

     \subsection{$G_\mu$ and the Kodaira-Spencer map}\label{KS} Recall that for any variation $\underline\sV = (\sV, \sF^i, \nabla)$ of (mixed) Hodge structures on $S$, Griffiths transversality $\nabla(\sF^i)\subset \sF^{i-1}$ shows that the (Gauss-Manin) connection $\nabla$ induces $\sO_S$-linear maps (the Kodaira-Spencer maps) $\theta^i: T_S\otimes gr^i \to gr^{i-1}$, where $gr^i = \sF^i/\sF^{i+1}$ are the graded pieces of the Hodge filtration (\cf \cite[1.3]{K}). In case $\underline\sV$ comes from an abelian scheme or a $1$-motive over $S$, the Hodge filtration has only two steps: $gr^1$ and $gr^0$, so that there is only one Kodaira-Spencer map (\cf also \cite[1.4]{A3}).
     
 Let us consider the family of $1$-motives  $[\Z \stackrel{1\mapsto \xi}{\to} A]  $ defined by $(A, \xi)$ \cite[X]{D1}: in this case, $\sV = \sH^1_{dR} ([\Z \stackrel{1\mapsto \xi}{\to} A] /S) $ is a vector bundle of rank $2g+1$, and $gr^1=\varOmega_A$.  The associated Kodaira-Spencer map is a morphism of vector bundles
     \begin{equation}\label{e25}\theta_\xi : \, T_S \otimes \varOmega_A \to gr^0,\end{equation} 
whose composition with the canonical projection $gr^0 \to Lie\, A$ is the usual Kodaira-Spencer map of $A$:
     \begin{equation}\label{e26}\theta_A  : \, T_S \otimes \varOmega_A \to   \varOmega_A^\vee = Lie\, A . \end{equation} 
     
     Conditions $iii), iv), v)$ above imply that the Hodge-theoretic graded piece $$gr^0 := gr^0\sH^1_{dR} ([\Z \stackrel{1\mapsto \xi}{\to} A] /S)$$ is a trivial extension of $Lie\, A$ by $\sO_S\,$. We denote by $\omega_0$ a lifting of $1  $ in $(gr^0)^\vee$, so that $(\omega_0, \omega_1, \ldots, \omega_g)$ form a basis of sections of $(gr^0)^\vee$.

  Using the basis $(\omega_1, \ldots, \omega_g)$ and the dual basis (and Notation \eqref{e19}), it turns out that the matrix of the contracted map $\theta_{A,\partial }: \varOmega_A \to   \varOmega_A^\vee$ is $T_{\partial }$ (which is a symmetric matrix), \cf \cite[1.4]{A3}. Similarly, the matrix of the contracted map $\theta_{\xi,\partial }:  \varOmega_A \to gr^0$ is $\begin{pmatrix} \ell_{\partial } \\ T_{\partial } \end{pmatrix}.$
   It follows that $G_\mu$ is the matrix of the contracted map 
  \begin{equation}\label{e27}
  \theta_\xi^{\varpi}:\, T_S \to   \varOmega_A^\vee = Lie\, A, \;\;\; \varpi = \sum_{k= 0}^g \mu_k \omega_k\in (gr^0)^\vee, \; \mu_0= 1 .
  \end{equation}
 Combining \eqref{e6}, \eqref{e18'} and \eqref{e23}, one finally gets
           \begin{equation}\label{e30}  J({\tilde s}) \cdot \begin{pmatrix}  \Omega , \overline\Omega \end{pmatrix}({\tilde s}) =
         \begin{pmatrix}    \theta_\xi^{\varpi(\tilde s)}(s)  &  0 \\ 
         0& \overline  \theta_\xi^{\varpi(\tilde s)}(s)\end{pmatrix}  \end{equation}
with $\varpi(\tilde s) = \omega_0 +  \sum_{k= 1}^g (M   - \Lambda N  + 2\pi i  \beta_2 \, {}^t\Omega_1^{-1})_k({\tilde s})\cdot \omega_k ,$ 
and  
 $\,\beta_2$ given by \eqref{e12'}; whence 
 \begin{equation}\label{e30'}  {\rm{rk}}\,J(\tilde s) \leq 2 \max_{\varpi} {\rm{rk}}\,  \theta_\xi^{\varpi}(s). \end{equation} 
 Combining  \eqref{e7}, \eqref{e18''} and \eqref{e24}, one also gets
  \begin{equation}\label{e31} \displaystyle    \max_{\varpi} {\rm{rk}}\,  \theta_\xi^{\varpi}(s) \leq  \max_{\nu\in \C^{2g}} {\rm{rk}}\,I_\nu({\tilde s}) .   \end{equation}

   \section{Generic rank of the Betti map.}

 In this section, we show that the generic rank of the (derivative of the) Betti map is twice the generic rank of the contracted Kodaira-Spencer map $\theta_\xi^\varpi$ for a {\it generic} value of the parameter $\varpi$. Here, the algebraicity of $S$ is essential, as well as the assumptions of Th. \ref{rk1}:  
 
\smallskip $(*)\;\;\;$ {\it $\Z \xi$ is Zariski-dense in $A$, and  $A$ has no fixed part}. 
  
  \subsection{A strengthened form of Theorem \ref{rk2}.}
  
  \begin{thm}\label{rk+} Under $(*)$,     
  \begin{equation}\label{e32} 
      \displaystyle\,{\rm{rk}} \, J \,  = 2\cdot \,\max_{\nu \in \C^{2g}} {\rm{rk}} \,I_{\nu}. 
      \end{equation}   
   \end{thm} 
Here, the ranks are the {\it generic} ranks.  
This implies the formula $\displaystyle{\rm{rk}} \, J \,= 2  \max_{\varpi} {\rm{rk}}\,  \theta_\xi^{\varpi}$ of Theorem \ref{rk2}, due to \eqref{e30'} and \eqref{e31}.  This also implies Theorem \ref{rk1}, thanks to \eqref{e14} (or \eqref{e18''}).

 Inequality $\leq$ in \eqref{e32} follows from \eqref{e7}. 
 If ${\rm{rk}} \,  \beta$ takes its maximal possible value $2\min (d,g)$, we have equality since ${\rm{rk}} \,I_{\nu}\leq \min (d,g)$. Therefore we may and shall assume that 
 
\smallskip $(**)\;\;\;$ $r:= {\rm{rk}} \, J  <  2\min (d,g)$.
  
 Equivalently (by \eqref{e7}): {\it every minor $\sM_{-\beta}$ of order $r+1$ in $I_{-\beta}$ vanishes identically on $\tilde S$}. We have to show that every minor $\sM_\nu$ of order $r+1$ in $I_{\nu}$ vanishes identically on $\tilde S$ for every $\nu \in \C^{2g}$.

    \subsection{Separation of the holomorphic and antiholomorphic variables.}\label{sepv}  By Remark \ref{remI} $(1)$,  $\det (\Omega, \overline\Omega)^r \sM_{-\beta}$ is a polynomial in holomorphic functions and antiholomorphic functions on $\tilde S$, \ie lies in the image of the ``diagonal" ring homomorphism 
    \begin{equation}\label{e34} \iota: \sO(\tilde S) \otimes \bar\sO(\tilde S) \to \sO(\tilde S^{real}), \;\; f(z)\bar g(w)\mapsto  f(z)\bar g(z).\end{equation}  
     \begin{lemma} $\iota$ is injective.
     \end{lemma} 
    
    \proof If $\sum  f_i(z)\bar g_i(z) = 0$, then $\sum  f_i(z)  g_i(w) \in \sO(\tilde S \times \tilde S)$ vanishes on the real-analytic subvariety of $\tilde S \times \tilde S$ given by the equation $z = \bar w$, hence on the smallest complex-analytic subvariety which contains it, which is $\tilde S \times \tilde S$ (the computation being local, one may replace $\tilde S$ by a open subset of $\C^d$; looking at tangent spaces, the result follows by induction on $d$).\qed
    
In fact, $\det (\Omega, \overline\Omega)^r. \sM$ can be written as a polynomial $P$ in the components of $\Lambda, \partial\Lambda,  \Omega, \partial\Omega$ and their complex-conjugates. By the previous lemma, one has 
  \begin{equation}\label{e35} P(\Lambda(\tilde s), \partial\Lambda(\tilde s),  \Omega(\tilde s), \partial\Omega(\tilde s), \overline\Lambda(\tilde s'), \overline\partial\overline\Lambda(\tilde s'),  \overline\Omega(\tilde s'), \overline\partial\overline\Omega(\tilde s')) =0
  \end{equation}
  for any pair $(\tilde s, \tilde s')\in \tilde S^2$.

  \subsection{Taking advantage of the monodromy: from a single equation to a parametrized system of equations}\label{mono} 
     If we temporarily fix $\tilde s'$, 
\eqref{e35} becomes an identity in $\C[\Lambda, \partial\Lambda, \Omega, \partial\Omega]$. 
In particular, it is invariant under the monodromy group, which acts on $\Lambda$ and $\Omega$ by the formulas \eqref{e9'} (which are independent of $\tilde s'$). 
  
  According to \cite[Th. 3]{A1}, under assumption $(*)$\footnote{this assumption implies that the algebraic subgroup $\tilde U\subset  \C^{2g}$ of \loccit is $\C^{2g}$ itself. Note also that this theorem uses the fact that the variation of mixed Hodge structure attached to the 1-motive $[\Z \stackrel{\xi}{\to} A]$ is good (= admissible), which is proven in lemma 5 of \loccit when $S$ is a curve. One can reduce to this case by taking a sufficiently general curve cut on $S$, or alternatively, by using the fact that ``goodness" may be checked in dimension $1$, \cf also \cite[Th. 2.2]{W}.}, the kernel of the map of monodromy groups $\{ \begin{pmatrix}   1 &  \nu_\delta \\ 
         0& M_\delta\end{pmatrix}  \}_{\delta \in \Delta}  
         \to \{ M_\delta\}_{\delta \in \Delta} $ is Zariski-dense in $\C^{2g}$  
(this result can be interpreted as a theorem of linear independence of abelian logarithms with respect to periods, \cf also \cite[\S 1]{A2} for a more concise and transparent proof). 
 Therefore, if one replaces $(\Lambda, \partial \Lambda)$ by $(\Lambda + \nu \Omega, \partial \Lambda + \nu \partial \Omega)$ for any $\nu \in \C^{2g}$, \eqref{e35} still holds.
 
\smallskip Note that \eqref{e35} expresses the annulation of an (arbitrary) minor of $$\iota^{-1}( I_{-\beta}\cdot \det (\Omega, \overline\Omega))).$$  
If one replaces $(\Lambda, \partial \Lambda)$ by $(\Lambda + \nu \Omega, \partial \Lambda + \nu \partial \Omega)$, then (by \eqref{e2}) $\iota^{-1}(\beta)   $ becomes $\iota^{-1}(\beta) + \nu   \iota^{-1}((\Omega, \overline\Lambda)(\Omega, \overline\Omega)^{-1})$, and  $\iota^{-1}( I_{-\beta}\cdot \det (\Omega, \overline\Omega))) $ becomes 
 $\iota^{-1}( I_{\breve\nu}\cdot \det (\Omega, \overline\Omega))),$ 
with  
  $$ \breve\nu  = (0, \nu\, \overline\Omega(\tilde s')) (\Omega(\tilde s), \overline\Omega(\tilde s'))^{-1} .$$
Since $\breve\nu$ depends only on $g$ parameters (instead of $2g$), this is not enough to conclude.

Instead of keeping $\tilde s'$ fixed, we will use monodromy of the second factor as well (and simultaneously), \ie the action of $\pi_1(S\times S, (s,s'))\cong \pi_1(S,s)\times \pi_1(S',s')$ (up to complex-conjugation on the second factor, this amounts to applying \cite[Th. 3]{A1} in the product situation $[\Z^2 \to A^2]/S^2$).  
 On the second factor, we have to replace $(\overline\Lambda, \overline\partial \overline\Lambda')$ by $(\overline\Lambda + \nu' \overline\Omega, \overline\partial \overline\Lambda + \nu' \overline\partial \overline\Omega)$ for any $\nu' \in \C^{2g}$. Then $\iota^{-1}( I_{-\beta}\cdot \det (\Omega, \overline\Omega)) $ becomes 
 $\iota^{-1}( I_{\breve\nu}\cdot \det (\Omega, \overline\Omega)),$ with  
  \begin{equation}\label{e36} \breve\nu = (\nu'\,\Omega(\tilde s),\; \nu\, \overline\Omega(\tilde s'))(\Omega(\tilde s), \overline\Omega(\tilde s'))^{-1} .\end{equation} 
For fixed $(\tilde s, \tilde s')$, it is then clear that $\breve\nu$ can take arbitrary values in $\C^{2g}$ when $(\nu, \nu')$ varies in $\C^{4g}$. 
One concludes that every minor $\sM_{\breve\nu}$ of order $r+1$ in $I_{\breve\nu}$ vanishes identically on $\tilde S$ for every $\breve\nu \in \C^{2g}$ as wanted. This finishes the proof of Theorems \ref{rk+} and \ref{rk2}.\qed

\smallskip \noindent {\it Proof of Cor. \ref{cor1}. } In this situation $d\geq g$.  As we have seen, by \eqref{e14}, Theorem \ref{rk+} implies that 
$$ \displaystyle\,{\rm{rk}} \, J \,  = 2\cdot \,\max_{\mu \in \C^{ g}} {\rm{rk}} \,H_{\mu}. $$ If $J$ is not a submersion, we have ${\rm{rk}} \, J < 2g$, hence for any $\mu, \, {\rm{rk}} \,H_{\mu} < g$. This amounts to saying that for any $(\mu_0, \mu_1, \ldots, \mu_g)\in \C^{g+1}$, all $g$-$g$-minors of the matrix with entries $\mu_0 \partial_i L_j  +  \sum_k \mu_k  \partial_iZ_{kj}$ vanish. For $\mu_0= 0$, consider the anaytic map sending $\tilde s\in \tilde S$ to $Z(\tilde s).\mu\in \C^{g}$. Its (complex-analytic) rank is $<g$, hence its non-empty fibers are complex analytic subvarieties of dimension $\geq \dim S +1 -g$.\qed

  \begin{rems} $1)$ Assumption $(\ast)$ is essential for \ref{rk+}: for instance, it is known that $\theta_A= 0$ if and only if $A$ becomes constant on a finite etale covering (\cf \cite[1.4.2 ii]{A3}), and by the same argument, $\theta_\xi= 0$ (equivalently: $\theta^\varpi_\xi= 0$ for all $\varpi$) if and only if $(A, \xi)$ becomes constant on a finite etale covering; on the other hand, $J = 0$ whenever $\xi$ is a torsion section, even if $A$ is non isotrivial.

 \smallskip\noindent $2)$ When $d=g$, Theorem \ref{rk1} expresses the condition ``$\beta$ is nowhere submersive" in the form: for any $\tilde s$,  $ (\partial_i L_j(\tilde s))$ and $(\partial_i Z_{kj}(\tilde s) ),\,  k=1,\ldots, g,$ span a vector space of singular matrices. One recognizes a special case of the notoriously difficult Dieudonn\'e problem: describe vector spaces of singular matrices, \resp of singular symmetric matrices, \cf \cite{Lo}\footnote{we are grateful to M. Brion for this reference.}.

\smallskip\noindent $3)$   Given an abelian scheme $A/S$, the vector bundle $\Omega_A$ on $S$ is endowed with a $\mathcal{D}_S$-module structure by the Gauss-Manin connection. The Kodaira-Spencer map $\theta_A$ is the $\mathcal{O}_S$-linear map
 $T_S\otimes \Omega_A\to \mathcal{H}^1_{\mathrm{dR}}(A/S)/\Omega_A$
 induced by the Gauss-Manin connection. It is then natural to introduce and compare the following `ranks', as done in \cite{A3}:\smallskip
 
 $r=r(A/S)=\mathrm{rk}\, \mathcal{D}_S \Omega_A/\Omega_A$,\smallskip
 
 $r'=r'(A/S)=\mathrm{rk}\,\theta = \mathrm{rk}\, \mathcal{D}_S^{\leq 1}\Omega_A/\Omega_A$\smallskip

 $r''=r''(A/S)=\max_{\partial} \mathrm{rk}\, \theta_\partial$.
 
 \smallskip\noindent
 One always has $r''\leq  r'\leq r\leq g$. Usually, by rank of the Kodaira-Spencer map one means the integer $r'$, and one  says that the Kodaira-Spencer map is non-degenerate if $r'=g$, which amounts to the condition  \cite[Lemma 1.4.5]{A3}
\begin{equation}\label{E.degenerescenza}
\forall \omega ,\, \exists \partial, \quad \theta_{A,\partial} \,\omega \neq 0. 
\end{equation}
 On the other hand, for $d= g$, the hypothesis of our Cor. \ref{cor2}  reads:
\begin{equation}\label{E.deg-pisana}
\exists \omega,\,  \forall \partial ,\quad \theta_{A,\partial} \, \omega \neq 0.
\end{equation}
There is no implication between \eqref{E.degenerescenza} and \eqref{E.deg-pisana}  in either direction in general (but we shall see that the conditions are related if $g\le 3$). 
\end{rems}

           \medskip\section{The case $g\leq 3$.}\label{S.g=3} 
      
 In this section, we shall use simple linear-algebraic arguments to tackle the case $g\leq 3$, using Cor. \ref{cor2}. However, for $g=3$, such arguments don't quite suffice and we shall have to invoke a result from \cite{A3} in order to obtain a complete proof of Theorem \ref{tg=3}.
 
 \subsection{} We have to show that under the assumptions of Theorem \ref{tg=3}, condition \eqref{E.deg-pisana} holds. 
 We may replace $S$ by a smooth locally closed subvariety of its image under $\mu_A$ (in order to {\it reduce to the case $d= d_{\mu_A}= g\leq 3$}), and $A$ by the restriction to $S$ of the universal abelian scheme $\mathcal X_{g,n}$. 
      In this situation, for any $s\in S$, the map 
       \begin{equation}\label{eq10}   \, T_{S,s} \stackrel{\partial \to \theta_{A,\partial} (s)}{\to} {\rm{Sym}}^2\varOmega_{A_s}^\vee\end{equation}  induced by Kodaira-Spencer is injective (\cf \cite[2.1.2]{A3}). 
       
       \begin{lemma}\label{Lemma-Lin-Alg}  
Suppose $g\leq 3$ and let $W $ be a vector space of dimension $g$ of quadratic forms on $\C^g$, containing a non-degenerate quadratic form. There exists a vector on $\C^g$ which is not in the kernel of any non-zero quadratic form in $W$.
\end{lemma}

 \begin{proof}
 We focus on the case $g=3$, leaving to the reader the (easy) verification for $g\leq 2$ (in that case, the assumption that $W$ contains a non-degenerate form is in fact automatic).
 The content of the above lemma can be rephrased in geometric terms as follows: the projective plane $\P(W)\subset\P(\mathrm{Sym}^2 (\C^3)^\vee)$ is a two-dimensional linear system of conics in $\P^2$. By hypothesis, $\P(W)$ contains at least one smooth conic. The lemma asserts that there exists a point $p\in\P^2$ such that no conic of the linear system $\P(W)$ is singular at $p$.

Recall that singular conics are either pairs of distinct lines, corresponding to quadratic forms of rank $2$, or double lines, corresponding to quadratic forms of rank $1$, i.e. squares of linear forms. In our setting, we start by proving that the set of double lines in $\P(W)$ is finite. In fact, $\P(W)$ is generated by three conics, and if it contains three double lines, these must be in general position (otherwise every conic in $\P(W)$ would   contained  their intersection and would be singular at that point). But in that case, $W$ is generated by the squares of three independent linear forms, so in suitable coordinates $W$ would consist  of the space of diagonal matrices, and $\P(W)$ then contains exactly three double lines.

It follows that the set of points which belong to some double line belonging to $\P(W)$ is one-dimensional or empty. 
 Also, since the generic conic in $\P(W)$ (which is two-dimensional) is smooth, the set of pairs of lines in $\P(W)$ is one-dimensional; each such pair of lines has exactly one singular point. It follows that the set of points $p$ of the plane such that $p$ is singular for at least one conic in $\P(W)$ is one-dimensional, so there is a point $p$ outside this set, as wanted.
  \end{proof}

\subsection{} In order to apply this lemma to \eqref{eq10} and derive \eqref{E.deg-pisana}, it is enough to show that, under our assumption, at least one $\theta_{A,\partial} $ is of rank $g$, \ie the Kodaira-Spencer map is non-degenerate in the sense of \eqref{E.degenerescenza}. 

 For $g\leq 2$, this is automatic. For $g=3$, one has the following general theorem of \cite{A3}, which relies on the theory of automorphic vector bundles: \eqref{E.degenerescenza} holds if $A/S$ is of ``restricted PEM type" (\cf \cite[4.2.2]{A3}). To be ``of PEM type" means that the connected algebraic monodromy group is maximal with respect to the polarization and the endomorphisms, and ``restricted" means that if the center $F $ of ${\rm End}_SA \otimes \Q$ is a CM field, then $\varOmega_A$ is a free $F\otimes_\Q \sO_S$-module\footnote{\loccit focusses on the case when the generic fiber of $A$ is absolutely simple, but the result extends in a straightforward way to the case when all the factors of the geometric generic fiber of $A$ come from non constant abelian schemes of restricted PEM type (over some finite covering of $S$).}.

It is well-known that all cases with $g=3$ are of PEM type\footnote{the first case not of PEM type is Mumford family with $g=4$.}, and even of restricted PEM type except when $End_S A$ is an imaginary quadratic order (over any finite covering of $S$). This exceptional case is investigated in \cite[4.1.2]{A3}, and it turns out that $d_{\mu_A}= 2< g= 3$, which is ruled out by our assumptions.   \qed

 \begin{rem} One can follow the pattern of proof of  Lemma \ref{Lemma-Lin-Alg} for higher values of $g$. This works provided the web $W$ of quadratic forms satisfies the following condition: for each $g'=0,\ldots,g$, the algebraic subvariety of $W$ formed of quadratic forms of rank $\leq g'$ has dimension $\leq g'$.
  This is not automatic for $g>3$: for $g= 4$, consider the web of quadratic forms spanned by $x_0^2, x_0x_1, x_1^2, x_2x_3$, which is connected with a four-dimensional  family of abelian varieties not satisfying condition \eqref{E.deg-pisana}. This  family consists of the products $A\times E^2$, where $A$ is a principally polarized abelian surface and $E$ an elliptic curve. 
  \end{rem}

\section{The hyperelliptic case.}
  \subsection{} Hyperelliptic curves of genus $g\geq 2$ admit a plane model given by an equation of the form $y^2=f(x)$, where $f(x)\in\C[x]$ is a polynomial of degree $2g+1$ without repeated roots. Up to affine transformations on the variable $x$, one can suppose that two of the roots are $0$ and $1$.
It is natural to consider the base 
\begin{equation*}
S:=\{(s_0,\ldots,s_{2g-2})\in \C^{2g-1}\, |\, s_i\neq 0,1,\ s_i\neq s_j\ \mathrm {for}\ i\neq j \}\subset \C^{2g-1}
\end{equation*}
and associate to any point $s=(s_0,\ldots,s_{2g-2})$ the hyperelliptic curve  obtained as a smooth completion of the plane affine curve 
\begin{equation}\label{E.Y_s}
Y_s:\quad y^2= x(x-1)(x-s_0)\cdots(x-s_{2g-2})=:f(x).
\end{equation}

We then obtain a family $\pi:Y\to S$ of hyperelliptic curves.
To the curve $Y_s$ we associate its jacobian $A_s$,  thus obtaining a $(2g-1)$-dimensional family of abelian varieties, over the base $S$.
In this case, the map $\mu$ to the moduli space of principally polarized genus $g$ abelian varieties is generically finite.
\smallskip

We identify the tangent space at $Y_s$ to the Torelli locus (of jacobians) in $\sA_g$ with  the dual of the vector space of quadratic differentials on $Y_s$ (which in turn is isomorphic to the first cohomology space  $H^1(Y_s,T_{Y_s})$). 
 A basis for the space $\Gamma(Y_s,\Omega_{Y_s}^1)$ consists of the 1-forms $\omega_0,\ldots,\omega_{g-1}$ defined, for $j=0,\ldots,g-1$, by $\omega_j :=\frac{x^j \mathrm{d}x}{y}.$

The quadratic differentials form a vector space $\Omega^{\otimes 2}_{Y_s}$ of dimension $3g-3$.
The hyperelliptic involution $\iota=\iota_s$, sending $(x,y)\mapsto \iota(x,y)=(x,-y)$, acts on this space. The $2g-1$ quadratic differential forms  $\frac{x^j \mathrm{d}x^2}{y^2}, \;j=0,\ldots,2g-2,$
turn out to be $\iota$-invariant, and form a basis of the space  of even quadratic differentials. Note that this space has the same dimension as the moduli space of genus $g$ hyperelliptic curves.

The Kodaira-Spencer map, in the form of (1.14) in \cite{A3}, can be viewed as a linear map $\theta: T_S(s)\to\mathrm{Sym}^2(\mathrm{Lie}A_s)$ whose image in contained in the subspace of even vectors of $\mathrm{Sym}^2(\mathrm{Lie}A_s)$.  
Hence we see $\theta_\partial$ as a functional on the space of quadratic differentials.

\medskip

The following computation of the Kodaira-Spencer map is probably well-known to experts, but we could not find a reference.

\begin{prop}\label{T.explicitKS}
In the above setting, for each $i=0,\ldots,2g-2$, there exists a non-zero number $c_i=c_i(s)$, depending on the point $s=(s_0,\ldots,s_d)\in S$, such that  for all $j=0,\ldots, 2g-2$
\begin{equation}\label{E.explicitKS}
\theta_{\partial/\partial s_i}\left(\frac{x^j\mathrm{d}x^2}{y^2}\right)=
c_i {s_i}^{j}.
\end{equation}
Also, for all $j=0,\ldots,g-3$, $\theta_{\partial/\partial s_i}\left(\frac{x^j\mathrm{d}x^2}{y}\right)=0.$
In particular, the Kodaira-Spencer map  induces an isomorphism between the tangent space $T_S(s)$ to the base at a given point $s$ and the space of linear functionals on $\Gamma(Y_s,\Omega^{\otimes 2}_{Y_s})$ vanishing on the subspace of odd quadratic differentials.
\end{prop}

  The above proposition enables to deduce Thm. \ref{hyp}  from Cor. \ref{cor2}: 
due to the surjectivity of the natural map $\Gamma(Y_s,\Omega_{Y_s})^{\otimes 2}\to\Gamma(Y_s,\Omega_{Y_s}^{\otimes 2}) $, 
Prop. \ref{T.explicitKS} enables to find for each $s\in S$ a vector $\omega\in \Gamma(S,\Omega_{A})$ such that the rank of $\theta_A^\omega(s)$ is $g$.

\subsection{} The strategy leading to formula \eqref{E.explicitKS} is as follows.
 Given a derivation $\partial$ on the base $S$ (e.g. $\partial=\partial/\partial s_i$)  we can cover the total space of the fibration $\pi: \mathcal{Y}\to S$ by open sets $U_\alpha$ on which there exist derivations $\partial_\alpha\in\Gamma(U_\alpha,T_{\mathcal{Y}})$ such that $\pi_*(\partial_\alpha)=\partial$. On the intersections $U_\alpha\cap U_\beta\cap {Y}_s$ the differences  $\partial_\alpha-\partial_\beta$ are derivations on $Y_s$ and form a cocyle. One can write each difference $\partial_\alpha-\partial_\beta$ in $U_\alpha \cap U_\beta$ as $\tilde{\partial}_\alpha-\tilde{\partial}_\beta$ for meromorphic derivations $\tilde{\partial}_\alpha$ (resp. $\tilde{\partial}_\beta$) on $U_\alpha$ (resp. $U_\beta$).

Then given a quadratic differential, which can be written as $\omega_1\times\omega_2$, for meromorphic differentials $\omega_1,\omega_2$, 
and a point $P\in {Y}_s$, taking an index $\alpha$ such that $P\in U_\alpha$, we can calculate the residue at $P$ of the meromorphic $1$-form $\omega_1(\tilde{\partial}_\alpha)\cdot\omega_2$.  The sum over $P\in Y_s$
of these numbers is taken to be the value of $\theta_\partial(\omega_1\otimes\omega_2)$  
of the Kodaira-Spencer map at this quadratic differential $\omega_1\times\omega_2$.

\subsection{} 

Let us consider the open set $U_x$   in which the $x$-function (as appearing in equation \eqref{E.Y_s}) is regular and has non-zero differential; analogously we define $U_y$ to be the open set where $y$ is regular and has non-zero differential. Note that $(U_x\cup U_y)\cap Y_s$ contains all points of $Y_s$ but the unique point at infinity.

Then observe that $\xi:=x^g/y$ is a local parameter at infinity. Denote by $U_\xi$ a neighborhood of the point at  infinity where $\xi$ is regular and its differential is non-zero; we can also choose $U_\xi$ so that it does not contain any point of $(U_x\cup U_y)\setminus (U_x\cap U_y)$, so that in $U_\xi\setminus\{\infty\}$ the three functions $\xi,x,y$ are regular with non-zero differential.
\smallskip

We then define the derivations $\partial_{i,x},\partial_{i,y},\partial_{i,\xi}$ as the unique derivations on the corresponding open sets extending  the derivation $\frac{\partial}{\partial s_i}$ on $S$ and such that $\partial_{i,u} u=0$ for $u\in\{x,y,\xi\}$. 
 We want to calculate the cocycle formed by the differences. We then calculate the values of $\partial_x,\partial_y,\partial_\xi$ on the two functions $x,y$ on $Y_s$.
 Using the equation for $Y_s$ and the fact that $\partial_x x=\partial_y y =0$ we find (omitting for simplicity the index $i$):
\begin{equation*}
\partial_x y =\frac{-y}{2(x-s_i)},\qquad \partial_y x = \frac{y^2}{h(x)(x-s_i)}.
\end{equation*}
Here the rational function $h(x)$ is defined as
\begin{equation*}
h(x)=\frac{f(x)}{x}+\frac{f(x)}{x-1}+\sum_{i=0}^{2g-1} \frac{f(x)}{x-s_i}
\end{equation*}
and satisfies $2y\mathrm{d}y=h(x)\mathrm{d}(x).$ Note that $h(x)$ is non-zero on $U_y$.
 Using that $\partial_\xi \xi=0$ and again the equation for $Y_s$, we obtain
\begin{equation*}
\partial_\xi x=\frac{xy^2}{(xh(x)-2g y^2)(x-s_i)},\qquad \partial_\xi y=\frac{gy^3}{(xh(x)-2gy^2)(x-s_i)}
\end{equation*}

\noindent This gives explicit formulae for $\partial_{xy}:=\partial_x-\partial_y$, $\partial_{x\xi}:=\partial_x-\partial_\xi$ and $\partial_{y\xi}:=\partial_y-\partial_\xi$. From these formulae, we see that all of them are regular at infinity. Looking at $\partial_{xy}$ we get
\begin{equation}\label{E.cocycle}
\partial_{xy} x=\frac{-y^2}{h(x)(x-s_i)},\qquad \partial_{xy} y =\frac{-y}{2(x-s_i)}.
\end{equation}
Choosing $\tilde{\partial}_x\in\Gamma(U_x,T_{Y_s}\otimes \C(Y_s))$ to be the zero constant, and $\tilde{\partial}_y\in\Gamma(U_y,T_{Y_s}\otimes \C(Y_s))$ to be the meromorphic derivation on $U_y$ satisfying the above formulae \eqref{E.cocycle} (with a minus sign, so that $\partial_{xy}=\tilde{\partial}_x-\tilde{\partial}_y$)  we see that the only pole in $U_y$ lies at the point $P_i=(s_i,0)$ (recall that the zeros of $h(x)$ do not belong to $U_y$). Applying the quadratic differential   $\frac{\mathrm{d}(x)^2}{y^2}$ we obtain a meromorphic differential $1$-form with simple pole  at $P_i$. Letting $c_i$ be its residue, we obtain formula \eqref{E.explicitKS}.  

If we apply instead a quadratic differential of type $\frac{\mathrm{d}(x)^2}{y}$,  we obtain a  $1$-form which is regular everywhere on $U_y$.
This shows that for suitable bases of $T_S(s)$ and $\Gamma(\Omega^{\otimes 2}, Y_s)^\vee$, the matrix of the Kodaira-Spencer map has the form $(T,0)$, for a square matrix   $T$ of order $2g-1$, which is a Vandermonde matrix, hence non-singular.\qed
\bigskip

 \section{Abelian schemes with $End_S A=\Z$ and $d_{\mu_A} \geq g$.}
 
 \subsection{} In this section, we prepare the proof of Theorem \ref{lbrk} by reducing to the case of maximal monodromy, the proof of which is treated in detail in App. 2.
  
 \begin{thm}\label{T8} Let $A\to S$ be a principally polarized abelian scheme of relative dimension $g$, such that
 \begin{enumerate}
 \item  $A$ has no non-trivial endomorphism over any finite covering of $S$,   \item $d_{\mu_A} :=  \dim  Im( \mu_A : S \to \sA_g$) is at least $g$. \end{enumerate}
 
 Then the monodromy of $A\to S$ is Zariski-dense in $Sp_{2g}$. \end{thm}

\proof One uses the representation-theoretic classification of abelian varieties $\,B\,$ with $\,{\rm{End}}\, B = \Z\,$ given by Borovoi \cite{Bo} (which one applies to a general fiber $B$ of $A\to S$). He proves that the special Mumford-Tate group is $\Q$-simple, and more precisely has the form $G = Res_{F/\Q} H$ where $F$ is a totally real number field and $H$ is an absolutely simple $F$-group. Thus $G_\R$ decomposes as a product  $\Pi_1^m H_i $ of absolutely simple real groups (obtained from $H$ at the various real places of $F$), and it follows from Satake's classification that {\it only $H_1$ is non-compact} (up to permutation). Accordingly, $H^1(B, \C)$ decomposes as a tensor product $\otimes V_i$ where each $V_i$ is an irreducible symplectic $\C$-representation of $H_{i\C}$. They are conjugate under the Galois group of $F^{gal}/\Q$, and in particular have the same dimension $2d$,  so that \begin{equation} (2d)^m = 2g,\end{equation} and {\it $m$ is odd} (again by the symplectic condition). 

In order to prove the theorem, one may assume that that $A[n]$ is constant for some $n$, and relplace $S$ by the smallest special subvariety of $\sA_{g,n}$ containing $S$, and $A$ by the universal abelian scheme over it (\cf \eg App. 2. Lemma 2.6). The conclusion of the theorem amounts to saying that generic Mumford-Tate group is $GSp_{2g}$, or equivalently that the special Mumford-Tate group $G$ of a general fiber $B$ is $Sp_{2g}$. 

It is known that, taking the above notation, $H_{1}$ is actually a classical absolutely simple $\R$-group: in particular, $H_{1\C}$ belongs to one of the series $A_\ell, B_\ell, C_\ell, D_\ell$ (with $\ell\geq 2,3,1,4$ respectively\footnote{more convenient here than the usual convention $1,2,3,4$: but $A_1 = C_1$ and $B_2= C_2$.}), and accordingly the hermitian symmetric domain $X_1 = \tilde S$ belongs to one of the series $$X_{A_\ell}=  SU(r,\ell +1 -r)/ S(U(r) \times U(\ell +1 -r)),\;  r\in \{1, \ldots, \ell\},$$  $$X_{B_\ell} = SO(2,2\ell -1)/SO(2)\times SO(2\ell-1),$$   \centerline{ $X_{C_\ell} = \frak H_{\ell}, $ the Siegel space,} $$X_{D^{\R}_\ell}= SO(2,2\ell-2)/SO(2)\times SO(2\ell-2), $$ \centerline{``quaternionic version"  of the latter, $X_{D^{\mathbb{H}}_\ell} = SO(2\ell)^{\mathbb{H}}/U(\ell)$.}  
 Their real dimensions are respectively  \begin{equation} 2r(\ell+1-r),\;  2(2\ell -1),\; \ell(\ell+1) ,\; 4(\ell -1), \;\ell(\ell-1) .\end{equation}
For such $A\to S$, condition $(2)$ in the theorem means that $\dim_\R X_1 \geq 2g$, which becomes in each case:
\begin{equation} 2r(\ell+1-r)\geq (2d)^m,\; 2(2\ell -1)\geq (2d)^m,\; \ell(\ell+1)\geq (2d)^m,\end{equation} $$\; 4(\ell -1)\geq (2d)^m, \; \ell(\ell-1)\geq (2d)^m. $$

To relate $\ell$ and $d$, one needs one more piece of information about the representation $V_{1\C}$:  according to Deligne, it is minuscule \cite{D2}. The classification of minuscule representations shows that in case $A_\ell$, $V_{1\C}$ is a wedge power $\wedge^i V_{st} $ of the standard representation,
 in case $B_\ell$ the spin representation of dimension $2d = 2^\ell$, in case $C_\ell$ the standard representation $V_{st}$ of dimension $2d = 2\ell$, in case $D_\ell$ either the standard representation or the half-spin representations of dimension $ 2d = 2^{\ell-1}$.   

The only symplectic cases among them are: in case $A_\ell $, $\wedge^{(\ell+1)/2} V_{st} $ if $\ell  \equiv 1 (4)$, in case $B_\ell$ the spin representation if $\ell \equiv 1,2, 5, 6 (8)$, in case $C_\ell$ the standard representation $V_{st}$, and in case $D_\ell$ the half-spin representations if $\ell \equiv 2 (4)$  (cf. also Mustafin's table \cite{Mu}).  These conguences imply $\ell \geq 5$ in the $A_\ell $ and $B_\ell$ cases, and $\ell \geq 6$ in the $D_\ell$ case.
 
The above inegalities become
\begin{equation} 2r(\ell+1-r)\geq  \begin{pmatrix} \ell +1\\ (\ell+1)/2\end{pmatrix}^m,\; 2(2\ell -1)\geq  2^{\ell m} ,\end{equation}
$$\; \ell(\ell+1)\geq (2\ell)^m,\; 4(\ell -1)\geq  2^{(\ell-1)m} ,  \ell(\ell-1)\geq  2^{(\ell-1)m}, $$ with $m$ odd. 

Let us first consider the third inequality ($C_\ell$ case). Since $m$ is odd, it is fulfilled if and only if $m=1$. This corresponds to the usual universal family $\sX_{g,n} \to \sA_{g, n}$ with $g= \ell$ and $G= Sp_{2g}$. 

\smallskip We claim that the other cases are impossible. For this, it suffices to consider $m=1$. A simple inspection shows that the first two inequalities are impossible for $\ell\geq 5$ (and any $r$), and the last two ones are impossible for $\ell \geq 6$. \qed

\subsection{Proof of Theorem \ref{lbrk}.}  Theorem \ref{T8} reduces the proof of Theorem \ref{lbrk} to the  special case where {\it the Zariski-closure of the monodromy of the scheme $A\to S$ is the full symplectic group $Sp_{2g}$.} In this situation, Gao's Theorem  \ref{ThmGeomCriterionHodgeGeneric} shows that our Corollary \ref{cor1} applies, and this ends the proof of Theorem \ref{lbrk}.
\qed

   \subsection{An example.} Beyond the case $End\, = \Z$, even for abelian scheme with simple geometric generic fiber, the condition in Cor. \ref{cor2} may fail, so that for any section $\xi$, our method thus fails to  establish whether $\beta$ is generically a submersion.
   
   Here is an example.    Let $L^+$ be a real quadratic field and $L$ a totally imaginary quadratic extension of $L^+$.  Let $A\to S$ be a principally polarized abelian scheme of relative dimension $16$, with level $n\geq 3$ structure, complex multiplication by $\sO_L$, and Shimura type $(r_\nu, s_\nu) = (0,8), (4,4)$ (for the two embeddings $\nu_1, \nu_2$ of $F^+$ in $\C$), and level $n$ structure. In the universal case, the base is a Shimura variety of PEL type of dimension $\sum r_\nu \cdot s_\nu = 16$.  
   
     By functoriality, $\theta_A$ commutes with the $\sO_L$-action, hence respects the decomposition $\mathcal H^1_{dR}= \oplus_{\rho: L\to \C}\, \mathcal H^1_{dR \rho}$ and the restriction of $\theta_{A,\rho}$ to the summands for $\rho $ above $\nu_1$ satisfy $\theta_{A,\rho}=0$ since $rk\, \varOmega_{A } \cap \mathcal H^1_{dR \rho}= 0  $ or $8$. Therefore $\max_\omega {\rm{rk}}\,\theta_A^\omega \leq 8 < 16$.

\bigskip
\section{Appendix I: the real hyperelliptic case.}\label{real}

In the present Appendix we  illustrate  a result related to the Betti map for a certain section of the family of Jacobians of  hyperelliptic curves of given genus, restricting however  to curves and points over $\R$. This  leads to issues of different kind compared to the purely complex case, and we shall present a treatment quite independent of  the rest of the paper. 
 
The issues are relevant concerning the paper \cite{Rob} by R.M. Robinson, as was noted by Serre. An independent  proof  was given recently by B. Lawrence in \cite{L}. Our purpose is to give an alternative argument depending on the Betti map. 
 
 It is a pleasure to thank Professor Serre for several helpful  comments, in particular pointing out inaccuracies in previous versions. He also suggested the present result, more general than a former one. 
 
\subsection{The relevant family of curves and Jacobians}  For $s=(s_0,\ldots ,s_{2g+1})\in \C^{2g+2}$, let $f_s(x)=x^{2g+2}+s_{2g+1}x^{2g+1}+\ldots +s_1x+s_0$, let $\Delta(s)$ be its discriminant, and consider   the open set $S\subset \C^{2g+2}$ consisting of the points with $\Delta(s)\neq 0$.

 We  consider the family $\pi:\J\to S$  of  the Jacobians $J_s$, of curves $Y_s$, $s\in S$, where $Y_s$ is a smooth complete curve birationally equivalent to the affine plane curve defined by 
\begin{equation}
y^2=f_s(x). 
\end{equation}
Since $\Delta(s)\neq 0$, this affine  curve is smooth, and $Y_s$ is obtained by adding  the two poles of the function $x$. We denote them by $\infty^\pm$, where the sign can be specified e.g. by stipulating that $y-x^{g+1}$ has a pole at $\infty^+$ of order $\le g$ (hence a pole at $\infty^-$ of order precisely $g+1$).  

The curve $Y_s$ has genus $g=\dim J_s$. We  consider the section $\sigma$ (of $\pi$) on $S$ to $\J$ such that $\sigma(s)= $ class of $[\infty^+]-[\infty^-]$ in $J_s$, for $s\in S$. For real $s\in S\cap \R^{2g+2}$, the points $\infty^\pm$ are defined over $\R$ and so is $\sigma(s)$.   
 
We want to present a proof of   the following

\begin{thm}\label{T.1}
The set of real $s\in S\cap\R^{2g+2}$ such that $\sigma(s)$ is a torsion point of $J_s$ is dense in $S\cap \R^{2g+2}$ (for the Euclidean topology).
\end{thm}

A main issue here is that we are concerned with {\it real points}  of the base $S$. The argument below 
uses also information coming from  complex points, and in this transition  {\it real -   complex} some attention has to be adopted.
For clarity we  recall again the Betti map for this case.

\subsection{Periods, abelian logarithms and the Betti map}   
There is a locally finite covering of $S$ 
by  
open  polydisks  $U_\alpha\subset S$, $\alpha\in I$,  
such that the holomorphic vector bundle over $S$ given by  the tangent  spaces to the  $J_s$ 
at the origin becomes  holomorphically equivalent to $U_\alpha\times \C^g$ over each $U_\alpha$.  Through this trivialization, for each $\alpha$ there are column vectors of analytic functions   $\omega_{1\alpha},\ldots ,\omega_{2g\alpha}:U_\alpha\to \C^g$ with the property   that for $s\in U_\alpha$ their values at $s$ form a basis for a lattice $\Lambda_{\alpha,s}\subset \C^g$ such that  the torus $T_s=\C^g/\Lambda_{\alpha,s}\cong J_s$  analytically  (through an exponential map $\exp_{\alpha, s}$). 


Note that  $U_\alpha\cap \R^{2g+2}$ is connected for each $\alpha\in I$ and  $U_\alpha\cap U_\beta$ is connected for any $\alpha,\beta\in I$ (if only because the $U_\alpha$ are convex).

\subsubsection{Subtori of real points.} 
For real $s\in U_\alpha\cap \R^{2g+2}$, the   lattice $\Lambda_{\alpha,s}$  has a sublattice $\Lambda'_{\alpha,s}$ of rank $g$, spanning over $\R$ a   vector space $V_{\alpha,s}:=\R\Lambda'_{\alpha,s}$ of dimension $g$, such that the connected component of the identity in the group  $J_s(\R)$ of real points of $J_s$  
corresponds to $V_{\alpha,s}/\Lambda'_{\alpha,s}$.  It is   known that   $J_s(\R)$  is  a finite union of translates, by torsion points  of order $2$, of this component; 
see e.g.   Prop. 1.1 in the paper \cite{GH} by B. Gross and J. Harris.

It is not difficult to see that we may also choose the above bases so  that $\omega_{1\alpha}(s),\ldots ,\omega_{g\alpha}(s)$ span $V_{\alpha,s}$ over $\R$  for $s\in U_\alpha\cap \R^{2g+2}$.

\subsubsection{Transition functions.} Let $\omega_\alpha$ denote the $g\times 2g$ matrix whose columns are the $\omega_{i\alpha}$. Then on the intersection $U_\alpha\cap U_\beta$ of two of the above open sets we have a transition expressed by
\begin{equation}\label{E.transition}
\omega_{\beta}=L_{\beta\alpha}\cdot \omega_\alpha \cdot R_{\alpha\beta},
\end{equation}
where $L_{\beta\alpha}\in GL_g(\sO_{U_\alpha\cap U_\beta})$  expresses  a change of basis of $\C^g$, and where $R_{\alpha\beta}\in GL_{2g}(\Z)$  expresses a change of basis for the lattice $\Lambda_{\alpha,s}$. 

\subsubsection{Abelian logarithms.} Since  the $U_{\alpha}$  are simply connected, 
an abelian logarithm $\lambda_\alpha$ of $\sigma$ may be defined on each  $U_\alpha$ as an analytic function to $\C^g$. It satisfies $\exp_{\alpha,s}\lambda_\alpha(s)=\sigma(s)$. For $s\in U_\alpha$, the value  $\lambda_\alpha(s)$ is uniquely determined up to a  vector in $\Lambda_{\alpha,s}=\omega_\alpha(s)\Z^{2g}$.
On an intersection $U_\alpha\cap U_\beta$ as above, such uniqueness and the above  transformations \eqref{E.transition} imply  
\begin{equation}\label{E.transition2}
\lambda_{\beta}=L_{\beta\alpha}\lambda_\alpha+ \omega_\beta v_{\alpha\beta},
\end{equation}
for suitable  integer column vectors $v_{\alpha\beta}\in \Z^{2g}$  (constant in $U_\alpha\cap U_\beta$).

\subsubsection{Betti coordinates and a restricted Betti map.}    For  $s\in U_\alpha$ 
we may write uniquely 
\begin{equation}\label{E.B}
\lambda_\alpha(s)=\sum_{i=1}^{2g}B_{i\alpha}(s)\omega_{i\alpha}(s)=:\omega_\alpha(s)B_\alpha(s),
\end{equation}
for  real numbers $B_{i\alpha}(s)$ (and a column vector $B_\alpha(s)\in\R^{2g}$). On taking complex conjugates, we see that this gives rise to  real functions   $B_{i\alpha}$ on $U_\alpha$ 
which are real-analytic in the variable $s\in U_\alpha\subset  \C^{2g+2}\cong\R^{4g+4}$. 
The above transition transformations \eqref{E.transition}, \eqref{E.transition2} yield,  on $U_\alpha\cap U_\beta$, 
\begin{equation}\label{E.B2}
B_\beta= R_{\alpha\beta}^{-1}B_\alpha +v_{\alpha\beta}.
\end{equation}

By definition  we have the (real-analytic) Betti map 
$$
B_\alpha:U_\alpha\to \R^{2g},\qquad B_\alpha(s)= \ ^t(B_{1,\alpha}(s),\ldots ,B_{2g,\alpha}(s)).
$$

\medskip

Since for real $s\in A\cap \R^{2g+2}$ the   points  $\infty^\pm$  lie in $Y_s(\R)$, the values $\lambda_\alpha(s)$  of $\lambda_\alpha$ on $U_\alpha\cap \R^{2g+2}$ lie, modulo $\Lambda_{\alpha,s}$, in one of the above mentioned finitely many translates of $V_{\alpha,s}$.  


 Now,  recalling that the translates in question are obtained by torsion points of order $2$,  
 this yields that  
 $B_\alpha(U_\alpha\cap \R^{2g+2})$  is inside $\R^g\times \{0\}+{1\over 2}\Z^{2g}$,  
 due to the present choice of $\omega_{1\alpha},\ldots ,\omega_{g\alpha}$. 
 By continuity and connectedness, the last $g$ coordinates of  the map are constant on $U_\alpha\cap \R^{2g+2}$.

 Then let us  now 
denote by $B_{\R\alpha}: U_\alpha\cap\R^{2g+2}\to \R^g$  the projection to the first $g$ coordinates of the map $B_\alpha$ restricted to real points; this is still real-analytic.

This really is meaningful only if the relevant set is non-empty, so, also for later reference, we define  $I_0\subset I$ as the subset of  the $\alpha\in I$ such that $U_\alpha\cap \R^{2g+2}$ is non-empty.

 \subsection{The differential of the restricted Betti map}  Let $s\in U_\alpha$; then   the value $\sigma(s)$ is torsion on $J_s$ if and only if  $B_\alpha(s)\in \Q^{2g}$. Then, to prove Theorem \ref{T.1}  it will suffice to prove
  \begin{thm}\label{T.2}
 For each $\alpha\in I_0$ the map $B_{\R\alpha}$  on  $U_\alpha\cap\R^{2g+2}$   
 attains rational values on a dense subset of $U_\alpha\cap \R^{2g+2}$. 
\end{thm}

For this,  the first thing we want to  show  is that its differential has maximal rank $g$ at some real point $s\in U_\alpha\cap \R^{2g+2}$ in case this set is non-empty, and in turn we shall prove this by comparison with the dimension of the fibers.  We start with:

\begin{prop}\label{P.1}  Suppose that for a certain $\alpha\in I_0$ and every point $s\in U_\alpha\cap \R^{2g+2}$  the differential of $B_{\R\alpha}$ has  rank $<g$ at $s$. Then,  for  all points $u\in U_\alpha\cap \R^{2g+2}$ the fiber  $B_\alpha^{-1}(B_\alpha(u))\subset U_\alpha$ is a complex variety in $U_\alpha$ of  complex dimension $>g+2$.
\end{prop}

\begin{proof}  
The assumption implies, through the (real) implicit function theorem,  that   at all points $u$  in an open dense subset of  $U_\alpha\cap \R^{2g+2}$ the    fiber $B_{\R\alpha}^{-1}(B_{\R\alpha}(u))$ of $B_{\R\alpha}$ on $U_\alpha\cap \R^{2g+2}$  has  real dimension $>g+2$. So, the complex fiber $B_\alpha^{-1}(B_\alpha(u))\subset U_\alpha$ at such  real points  $u$ 
has  a real subset (that is, a subset inside $U_\alpha\cap \R^{2g+2}$) of real dimension $>g+2$, since  this complex fiber  contains   the respective fiber   $B_{\R\alpha}^{-1}(B_{\R\alpha}(u))\subset U_\alpha\cap \R^{2g+2}$.  
 
 But the fibers of $B_\alpha$  (on $U_\alpha$) are  complex varieties inside $U_\alpha$; in fact, the fiber  $B_\alpha^{-1}(b_1,\ldots ,b_{2g})$ is defined in $U_\alpha$  by the complex-analytic equation $\lambda_\alpha(z)=\sum_{i=1}^{2g}b_i\omega_{i\alpha}(z)$. We deduce that  {\it the complex dimension at $u$ of the   fiber   $B_{\alpha}^{-1}(B_\alpha(u))$ (as a subset of $U_\alpha$)  is $>g+2$ for all $u$ in a suitable open dense subset of  $U_\alpha\cap \R^{2g+2}$}.

The sought conclusion of the Proposition now  follows immediately  from the   following result, used also in the recent paper  \cite{CMZ}, to which we refer for a proof:

{\it For every $k\in\N$, the set $\{s\in U_\alpha: \dim_s B_\alpha^{-1}(B_\alpha(s))\ge k\}$ is   closed in $U_\alpha$.}
\end{proof}



Next, we want to extend  the conclusion of the proposition (keeping the same assumption) to a whole connected component of $S\cap \R^{2g+2}$, namely proving the following sharpening:

\begin{prop} \label{P.2} Suppose that for some $\alpha\in I_0$ and every point $s\in U_\alpha\cap \R^{2g+2}$  the differential of $B_{\R\alpha}$ has  rank $<g$ at $s$. Then  for every $u\in A\cap \R^{2g+2}$   lying in the same connected component of $U_\alpha\cap \R^{2g+2}$,  and for every $\gamma$ such that $u\in U_\gamma$, the  fiber $B_\gamma^{-1}(B_\gamma(u))\subset S$ has complex dimension $>g+2$ at $u$.
\end{prop}

\begin{proof} 

 Let $I_0'$ be the subset of indices $\beta\in I_0$ such that the differential of $B_{\R\beta}$ has rank $<g$ everywhere on $U_\beta\cap\R^{2g+2}$. We contend that  the union $\bigcup_{\beta\in I_0'}(U_\beta\cap \R^{2g+2})$ is open and closed in $A\cap \R^{2g+2}$, hence a union of connected components of $S\cap\R^{2g+2}$.


Now, this union is open since every $U_\alpha$ is open. To prove it is closed in $S\cap \R^{2g+2}$,   let $\gamma\in I_0'$ and let $\delta\in I_0$ be such that $U_\gamma\cap U_\delta\cap \R^{2g+2}$ is non-empty.  

Note that the transformation \eqref{E.B2} (and our convention about the choice and orderings of the bases $\omega_\gamma, \omega_\delta$) yields that  $B_{\R\delta}=\Gamma_{\gamma\delta}B_{\R\gamma}+{1\over 2}\tilde v_{\gamma\delta}$ on  $U_\gamma\cap U_\delta\cap \R^{2g+2}$, for a suitable  integral $g\times g$ matrix $\Gamma_{\gamma\delta}$ 
and a suitable $\tilde v_{\gamma\delta}\in\Z^g$.  

Then the  differential of $B_{\R\delta}$  shall have rank $<g$ on  $U_\gamma\cap U_\delta\cap \R^{2g+2}$, hence on the whole $U_\delta\cap \R^{2g+2}$, for the former set is non-empty and open in the latter,   the latter is connected and the relevant map is real-analytic.  In other words,  $\delta\in I_0'$ too.  

The closure of  $\bigcup_{\beta\in I_0'}(U_\beta\cap \R^{2g+2})$  follows: let  $s\in A\cap\R^{2g+2}$ lie  in the complement (with respect to $S\cap\R^{2g+2}$); this  $s$ lies in some $U_\eta\cap \R^{2g+2}$ and then, by what we have seen, $U_\eta\cap\R^{2g+2}$ shall be disjoint from our set. Therefore the complement of our set is open, as asserted.

\medskip

It follows that if  
$\alpha\in I_0'$, then $\gamma\in I_0'$ for every $\gamma\in I_0$ such that $U_\gamma\cap \R^{2g+2}$ is in the same connected component of $U_\alpha\cap \R^{2g+2}$ (relative to $S\cap \R^{2g+2}$).

Now, Proposition \ref{P.2} follows on applying   Proposition \ref{P.1} to each such $\gamma$. 
\end{proof}


\subsection{Contradicting the conclusion of Proposition \ref{P.2}} To contradict the conclusion of this proposition, hence proving that  its assumption cannot  hold, let us first inspect the connected components of $S\cap \R^{2g+2}$. Each point  $s\in A\cap \R^{2g+2}$ corresponds to a real monic polynomial  of degree $2g+2$ with no multiple complex root.  Let $2r(s)$ be the number of real roots, so $r(s)$ is an integer, $0\le r(s)\le g+1$. 
It is   easy to see that $r(s)$ is locally constant in $S\cap \R^{2g+2}$   and then it readily follows   that each (non-empty) connected component is defined in  $S\cap \R^{2g+2}$ by an equation $r(s)=r$, where $r$ is a given integer in $[0,g+1]$.  

We now have the following elementary  lemma, whose proof we leave to the interested reader:

\begin{lemma}
For each $r\in\{0,1,\ldots ,g+1\}$ there exist  a monic polynomial $P\in \R[x]$ of degree $g+1$ and a  real number $p\neq 0$ such that $P(x)^2-p$ has precisely $2r$ simple real roots and no multiple complex root.
\end{lemma}

Let now  $\alpha\in I_0$ be such that the assumption of  Proposition \ref{P.2} holds for $U_\alpha\cap \R^{2g+2}$, and let $K$ be the  connected component  of $U_\alpha\cap \R^{2g+2}$ in  $S\cap \R^{2g+2}$. In the above description, $K$ corresponds to an integer $r\in [0,g+1]$. Let us apply the lemma to it and let $P(x), p$ the polynomial and  real number coming from the lemma.  Then the polynomial  $P(x)^2-p$  corresponds, in our opening notation, to a point $s\in K$, so that  $f_s(x)=P(x)^2-p$.


The identity  $P(x)^2-f_s(x)=p\neq 0$ shows that the function $P(x)+y$  on $Y_s$ 
has divisor $(g+1)([\infty^+]-[\infty^-])$, hence 
$\sigma(s)$  is torsion in $J_s$, of   order (dividing but in fact equal to) $g+1$. Thus, for every $\gamma$ with $s\in U_\gamma$, the point  $\rho:=B_\gamma(s)$  is a rational point in $\Q^{2g}$ with denominator dividing $g+1$.

Let us look at the complex  fiber $B_\gamma^{-1}(\rho)$ around $s$: it consists of complex points $t\in S$ such that  
$\sigma(t)$ has torsion order dividing $g+1$ on the Jacobian $J_t$ of  $Y_t$. 
Hence for such a $t$  there is a rational function on $Y_t$ with divisor $(g+1)([\infty^+]-[\infty^-])$; this corresponds  to a  polynomial $P_t$ of degree $g+1$ such that  $P_t(x)^2=f_t(x)+p_t$ for some nonzero complex number  $p_t$. This $P_t(x)$  
 is monic of  degree $g+1$,  hence  
 it     depends on $\le g+1$ complex parameters; taking also $p_t$ into account, we see that  the complex dimension of the   said   fiber cannot be $>g+2$, yielding the sought contradiction. 

\subsection{Conclusion of the argument}  We have thus proved that for  every  $\alpha\in I_0$ there exists some  point $s\in U_\alpha\cap \R^{2g+2}$  such that  the differential of $B_{\R\alpha}$ has  rank $\ge g$, hence maximal rank $g$,  at $s$. 
  Then such differential  has maximal rank on  a dense open subset $U_\alpha'$ of (real)  points in $U_{\alpha}\cap \R^{2g+2}$. At this stage we can  invoke the following general fact:
 
{\it Let  $f : X\to Y$  be a continuous map between topological spaces; suppose there is an open dense subset  $X'$  of  $X$  such that $f_{|X'} : X' \to Y$ is open. Let  $Z$  be a dense subset of  $Y$. Then  $f^{-1}(Z)$  is dense in  $X$}.

{\it Proof }. Suppose not. Then, there is a non-empty open subset  $V$ of $ X$ such
  that  $ f(V) \cap Z$  is empty. Since $X'$ is dense,  $X' \cap V $  is not empty. Because $f_{|X'}$  is open, $ f(X' \cap V)$ is open in $Y$ and does not intersect $Z$ : this contradicts the fact that  $Z$  is dense in  $Y$.
  
  \medskip
  
  Now we may apply this statement to $X=U_\alpha\cap \R^{2g+2}$, $f=B_{\R\alpha}$ and to $X'=U_\alpha'$, taking for $Z$ the set of rational points in $B_{\R\alpha}(X)$. Theorem \ref{T.2} follows.

 \medskip

{\bf Remark} (i) One might ask about a $p$-adic analogue; it seems to us  that a density statement does not hold. 

\bigskip

\bigskip
\section{Appendix II by Z. Gao: An application of the (pure) Ax-Schanuel Theorem.}\label{Gao}
  
\subsection{Main Result}

Let $\sA_g$ be the moduli space of principally polarized abelian varieties of dimension $g$, possibly with some level structure. 

Let $S$ be a complex irreducible algebraic variety. Let $\pi{:}\; A\rightarrow S$ be an abelian scheme of relative dimension $g$. We may assume that $A/S$ is principally polarized up to replacing $A$ by an isogeneous abelian scheme. Then $\pi {:}\; A\rightarrow S$ induces a modular map $\mu_A {:}\; S\rightarrow \sA_g$. {\it  We assume $\dim \mu_A(S) \ge g$}.

 
Next we want to understand when $S$ satisfies the conclusion of Cor. \ref{cor1}. More precisely, let $\frak{H}_g$ be the Siegel upper half space and let ${u} :  \frak{H}_g \rightarrow \sA_g$ be the uniformization. Denote by $\tilde{S}$ a complex analytic irreducible component of ${u}^{-1}(\mu_A(S))$ in $\frak{H}_g$. We name

\smallskip\noindent\boxed{Condition~ACZ}
{\it For any $\tilde{s} \in \tilde{S}$ and any $\mathbf{c} \in \C^g$, there exists a complex analytic subvariety $\tilde{C} \subset \tilde{U}$ of dimension $\dim \mu_A(S) - g + 1$ passing through $\tilde{s} \in \frak{H}_g$ such that $\tilde{s}'  \mathbf{c}$ is constant for any $\tilde{s}' \in \tilde{C}$. Here we view $\mathbf{c} \in \C^g$ as a column vector and $\tilde{s}'  \mathbf{c}$ is the usual matrix product (recall that every point in $\frak{H}_g$ is a $g\times g$ matrix).}




Our main result is:
\begin{thm}\label{ThmGeomCriterionHodgeGeneric}
If Condition~ACZ is satisfied, then $\mu_A(S)$ is contained in a \textbf{proper} special subvariety of $\sA_g$.
\end{thm}
\begin{rem}
There are several equivalent ways to state the conclusion of Theorem~\ref{ThmGeomCriterionHodgeGeneric}. In fact for any irreducible subvariety $S$ of $\sA_g$, we shall prove in Lemma~\ref{LemmaEquivalentConditionsForProperSpecial} that the following statements are equivalent:
\begin{enumerate}
\item[(i)] The variety $S$ is not contained in any proper special subvariety of $\sA_g$.
\item[(ii)] The variety $S$ contains a point with Mumford-Tate group ${GSp}_{2g}$.
\item[(iii)] The monodromy group of $A/S$ is Zariski dense in ${Sp}_{2g}$.
 \end{enumerate}
\end{rem}
 In practice, condition (i) is often checked by computation of the Mumford-Tate group (hence condition (ii)) or of the monodromy group (hence condition (iii)). For example as $A/S$ is not isotrivial by Assumption~1, if we denote by $V = H^1(A_s,\C)$ for any $s \in S$, then condition (iii) holds if (and only if) the symmetric square $S^2 V$ is irreducible for the monodromy action by Beukers-Brownawell-Heckman \cite[Theorem~2.2]{BeukersSiegel-Nor}.

In fact, Theorem~\ref{ThmGeomCriterionHodgeGeneric} can be deduced from a more technical and more general theorem. In order to state the theorem we need to introduce some notation.

Recall the uniformization ${u} :  \frak{H}_g \rightarrow \sA_g$. The natural embedding $\frak{H}_g \subset \C^{g(g+1)/2}$ endows $\frak{H}_g$ with a structure of ``complex algebraic variety'', and hence ${u}$ gives rise to a bi-algebraic system. Say a complex analytic irreducible subset $\tilde{Y}$ of $\frak{H}_g$ is \textit{bi-algebraic} if $\tilde{Y}$ is algebraic in $\frak{H}_g$ and ${u}(\tilde{Y})$ is algebraic in $\sA_g$. See $\mathsection$\ref{SectionReviewOnAS} for more details.

Denote by $H_S^{\circ}$ the connected algebraic monodromy group of $A/S$, namely $H_S^{\circ}$ is the neutral component of the Zariski closure of $\mathrm{im}(\pi_1(S,s) \rightarrow \pi_1(\sA_g,s)) \subset {Sp}_{2g}(\Z)$ in ${Sp}_{2g}$. Denote by $\tilde{S}^{\mathrm{biZar}}$ the smallest bi-algebraic subset of $\frak{H}_g$ which contains $\tilde{S}$. It exists by Lemma~\ref{LemmaIntersectionOfBiAlgebraic}.

Recall that every element of $\frak{H}_g$ is a $g\times g$-matrix. For any $\mathbf{c} \in \C^g$ and any $\tilde{s} \in \frak{H}_g$, denote by
\[
H_{\mathbf{c},\tilde{s}} := \{Z \in \frak{H}_g: Z \mathbf{c} = \tilde{s}  \mathbf{c}\}.
\]

\begin{thm}\label{ThmStrongerForm}
There does not exist an abelian scheme $A/S$ with $\dim \mu_A(S) \ge g$ satisfying the following three properties:
\begin{enumerate}
\item[(i)] The connected algebraic monodromy group $H_S^{\circ}$ is simple;
\item[(ii)] There exist $\mathbf{c} \in \C^g$ and $\tilde{s} \in \tilde{S}$ such that $\mathrm{codim}_{\tilde{S}^{\mathrm{biZar}}}(H_{\mathbf{c},\tilde{s}} \cap \tilde{S}^{\mathrm{biZar}}) = g$.
\item[(iii)] Condition~ACZ is satisfied.
\end{enumerate}
\end{thm}

\subsection{Review on the bi-algebraic system of $\sA_g$ and Ax-Schanuel}\label{SectionReviewOnAS}
We focus on the case $\sA_g$. We shall consider the uniformization ${u} {:}\; \frak{H}_g \rightarrow \sA_g$, where $\frak{H}_g$ is the Siegel upper half space defined as the following.
\[
\frak{H}_g = \left\{ Z = X + \sqrt{-1}Y \in \mathrm{Mat}_{g\times g}(\C) : Z = Z^{\mathrm{t}},~ Y>0 \right\}.
\]
Later on we will study subvarieties of $\frak{H}_g$ and $\sA_g$ at the same time. To distinguish them we often use letters to denote subsets of $\sA_g$ and add a $\sim$ on top to denote subsets of $\frak{H}_g$.

Consider
\[
{\frak p}_g = \left\{ Z \in \mathrm{Mat}_{g\times g}(\C) : Z = Z^{\mathrm{t}} \right\}.
\]
Then ${\frak p}_g \cong \C^{g(g+1)/2}$ as $\C$-vector spaces. The Siegel upper half space is open (in the usual topology) and semi-algebraic in ${\frak p}_g$. The complex structure on ${\frak p}_g$ thus induces a structure of complex analytic variety on $\frak{H}_g$. Following the convention of Pila, Ullmo and Yafaev, we define
\begin{defn}
A subset $\tilde{Y}$ of $\frak{H}_g$ is said to be \textbf{irreducible algebraic} if it is a complex analytic irreducible component of $\frak{H}_g \cap \tilde{Y}^{\mathrm{c}}$ for some algebraic subvariety $\tilde{Y}^{\mathrm{c}}$ of ${\frak p}_g$.
\end{defn}
Hence we have the following definition.
\begin{defn}
\begin{enumerate}
\item A subset $\tilde{F}$ of $\frak{H}_g$ is said to be \textbf{bi-algebraic} if it is irreducible algebraic in $\frak{H}_g$ and ${u}(\tilde{F})$ is an algebraic subvariety of $\sA_g$.
\item An irreducible subvariety $F$ of $\sA_g$ is said to be \textbf{bi-algebraic} if one (and hence any) complex analytic irreducible component of ${u}^{-1}(F)$ is irreducible algebraic in $\frak{H}_g$.
\end{enumerate}
\end{defn}

Before moving on, let us make the following observation. 
\begin{lemma}\label{LemmaIntersectionOfBiAlgebraic}
Let $F_1$ and $F_2$ be two bi-algebraic subvarieties of $\sA_g$, and let $F$ be an irreducible component of $F_1 \cap F_2$. Then $F$ is also bi-algebraic.
\end{lemma}
\begin{proof}
First $F$ is clearly irreducible algebraic. Consider a complex analytic irreducible component $\tilde{F}$ of ${u}^{-1}(F)$. It is contained in both $\tilde{F}_1$, an irreducible component of ${u}^{-1}(F_1)$, and $\tilde{F}_2$, an irreducible component of ${u}^{-1}(F_2)$. Let $\tilde{F}'$ be an irreducible component of $\tilde{F}_1 \cap \tilde{F}_2$ which contains $\tilde{F}$. Then
\[
F = {u}(\tilde{F}) = {u}(\tilde{F}') \subset {u}(\tilde{F}_1) \cap {u}(\tilde{F}_2) = F_1 \cap F_2.
\]
Taking the Zariski closures, we get
\[
F \subset {u}(\tilde{F}')^{{Zar}} \subset F_1 \cap F_2.
\]
Now since $\tilde{F}'$ is irreducible, we know that ${u}(\tilde{F}')^{{Zar}}$ is irreducible. Hence $F = {u}(\tilde{F}')^{{Zar}}$ since $F$ is an irreducible component of $F_1 \cap F_2$. Therefore $F = {u}(\tilde{F}')$ and so $\tilde{F} = \tilde{F}'$ is algebraic. So $F$ is bi-algebraic.
\end{proof}

Based on this observation, for any irreducible subvariety $Y$ of $\sA_g$, there exists a unique smallest bi-algebraic subvariety of $\sA_g$ which contains $Y$. Then for any complex analytic irreducible subset $\tilde{Y}$ of $\frak{H}_g$, there exists a unique smallest bi-algebraic subset of $\frak{H}_g$ which contains $\tilde{Y}$: it is an irreducible component of the smallest bi-algebraic subvariety of $\sA_g$ which contains ${u}(\tilde{Y})^{{Zar}}$.

There is a better characterization of bi-algebraic subvarieties of $\sA_g$ using Hodge theory and group theory. They are precisely the so-called \textit{weakly special} subvarieties of $\sA_g$ defined by Pink \cite[Definition~4.1.(b)]{PinkA-Combination-o}. This is proven by Ullmo-Yafaev \cite[Theorem~1.2]{UllmoA-characterisat}. Moonen has also studied these subvarieties and proved that they are precisely the totally geodesic subvarieties of $\sA_g$. See \cite[4.3]{MoonenLinearity-prope}. Linearity properties in Shimura varieties was first studied by Moonen in \textit{loc.cit}. For our purpose we prove the following lemma.

\begin{lemma}\label{LemmaWeaklySpecialLinear}
Let $\tilde{F}$ be a bi-algebraic subset of $\frak{H}_g$. Then it is affine linear, meaning that it is the intersection of $\frak{H}_g$ with some affine linear subspace of ${\frak p}_g$.
\end{lemma}
\begin{proof} This follows from Ullmo-Yafaev's characterization and the Harish-Chandra realization of Hermitian symmetric domains. Let us explain the details. We use the language of Shimura data in the proof.

By a result of Ullmo-Yafaev \cite[Theorem~1.2]{UllmoA-characterisat}, bi-algebraic subsets of $\frak{H}_g$ are precisely the weakly special subsets of $\frak{H}_g$. Hence $\tilde{F}$ is a weakly special subset of $\frak{H}_g$. By definition of weakly special subvarieties (see \cite[Definition~2.1]{UllmoA-characterisat} or \cite[Definition~4.1.(b)]{PinkA-Combination-o}), there exist a connected Shimura subdatum $(G,\sX)$ of $({GSp}_{2g},\frak{H}_g)$ and a decomposition $(G^{{ad}},\sX) = (G_1,\sX_1) \times (G_2,\sX_2)$ and a point $\tilde{x}_2 \in \sX_2$ such that $\tilde{F} = \sX_1 \times\{\tilde{x}_2\}$.

Take any point $(\tilde{x}_1,\tilde{x}_2) \in \sX_1 \times \{\tilde{x}_2\} \subset \sX$, we have the Harish-Chandra embedding of $\sX$ into $T_{(\tilde{x}_1,\tilde{x}_2)} \sX$, the tangent space of $\sX$ at $(\tilde{x}_1,\tilde{x}_2)$. We have also the Harish-Chandra embedding of $\frak{H}_g$ into $T_{(\tilde{x}_1,\tilde{x}_2)} \frak{H}_g$. These two embeddings are compatible in the following sense: $T_{(\tilde{x}_1,\tilde{x}_2)} \sX$ is a linear subspace of $T_{(\tilde{x}_1,\tilde{x}_2)} \frak{H}_g$ and $\sX = \frak{H}_g \cap T_{(\tilde{x}_1,\tilde{x}_2)} \sX$. This is proven in Helgason \cite[Chapter~VIII, $\mathsection$7]{HelgasonDifferential-Ge}. We refer to \cite[Chapter~5, $\mathsection$2, Theorem~1]{MokMetric-Rigidity} for the presentation. In particular, the Harish-Chandra embedding realizes $\frak{H}_g$ as the unit ball in $\C^{g(g+1)/2}$.

The natural embedding of the Hermitian symmetric space $\frak{H}_g$ into ${\frak p}_g$ can be realized as the Harish-Chandra embedding mentioned above composed with a linear transformation which we call $\ell$. Define ${\frak p} = \ell(T_{(\tilde{x}_1,\tilde{x}_2)} \sX)$. Then ${\frak p}$ is an affine subspace of ${\frak p}_g$. Now by the compatibility mentioned in the last paragraph, we have that $\sX = {\frak p} \cap \frak{H}_g$.

The decomposition of Hermitian symmetric spaces $\sX = \sX_1 \times \sX_2$ gives a decomposition ${\frak p} = {\frak p}_1 \times {\frak p}_2$ as complex spaces, and $\tilde{F} = \sX_1 \times \{\tilde{x}_2\}$ is then $({\frak p}_1 \times\{0\}) \cap \sX = ({\frak p}_1 \times\{0\}) \cap \frak{H}_g$. Hence we are done.
\end{proof}

Now we are ready to state the Ax-Schanuel theorem for $\sA_g$. It is recently proven by Mok-Pila-Tsimerman \cite{MokAx-Schanuel-for}. This theorem has several equivalent forms, whose equivalences are not hard to show. For our purpose we only need the following weak form.
\begin{thm}\label{ThmAS}
Let $\tilde{Y}$ be an irreducible complex analytic subvariety of $\frak{H}_g$. Let ${\tilde{Y}}^{biZar}$ be the smallest bi-algebraic subset of $\frak{H}_g$ which contains $\tilde{Y}$. Then
\[
\dim \tilde{Y}^{{Zar}} + \dim {u}(\tilde{Y})^{{Zar}} \ge \dim \tilde{Y}  + \dim {\tilde{Y}}^{biZar}.
\]
Here $\tilde{Y}^{{Zar}}$ means the smallest irreducible algebraic subset of $\frak{H}_g$ which contains $\tilde{Y}$.
\end{thm}

We end this section by proving the equivalence of the following statements.
\begin{lemma}\label{LemmaEquivalentConditionsForProperSpecial}
Let $S$ be an irreducible subvariety of $\sA_g$. Then the following statements are equivalent:
\begin{enumerate}
\item[(i)] The variety $S$ is not contained in any proper special subvariety of $\sA_g$.
\item[(ii)] The variety $S$ contains a point with Mumford-Tate group ${GSp}_{2g}$.
\item[(iii)] The monodromy group of $A/S$ is Zariski dense in ${Sp}_{2g}$.
\item[(iv)] The variety $S$ is not contained in any proper bi-algebraic subvariety of $\sA_g$ of positive dimension.
\item[(v)] There exists a point $s \in S(\C)$ such that the following condition holds: $s$ is not contained in any proper bi-algebraic subvariety of $\sA_g$ of positive dimension.
\end{enumerate}
\end{lemma}
\begin{proof}
By Deligne-Andr\'{e}, a very general point in $S(\C)$ has the same Mumford-Tate group which we denote by $\mathrm{MT}(S)$. The Mumford-Tate group $\mathrm{MT}(S)$ is a reductive group, and a subgroup of finite index of the monodromy group is contained in $\mathrm{MT}(S)^{\mathrm{der}}$. Here ``very general'' means that the point is taken outside an at most countable union of proper subvarieties of $S$. We refer to \cite[Lemma~4]{A1} for these facts.

Now let us prove (iii) $\Rightarrow$ (i). If the monodromy group of $A/S$ is Zariski dense in ${Sp}_{2g}$, then ${Sp}_{2g}$ is a subgroup of $\mathrm{MT}(S)^{\mathrm{der}} < {GSp}_{2g}^{\mathrm{der}} = {Sp}_{2g}$. Hence $\mathrm{MT}(S)^{\mathrm{der}} = {Sp}_{2g}$. So $\mathrm{MT}(S) = {GSp}_{2g}$.\footnote{By Hodge theory, $\mathbb G_m = Z({GSp}_{2g})$ is contained in $\mathrm{MT}(S)$. In this paper we only need ${Sp}_{2g} < \mathrm{MT}(S)$.} 

For (ii) $\Rightarrow$ (iii), we use a stronger result of Andr\'{e}. Let $H_S^{\circ}$ be the neutral component of the Zariski closure of the monodromy group of $A/S$ in ${GSp}_{2g}$. Then by \cite[Theorem~1]{A1}, we have that $H_S^{\circ}$ is a non-trivial normal subgroup of ${GSp}_{2g}^{\mathrm{der}} = {Sp}_{2g}$. But ${Sp}_{2g}$ is simple, so $H_S^{\circ} = {Sp}_{2g}$.

Let us prove (i) $\Rightarrow$ (ii). The smallest special subvariety of $\sA_g$ which contains $S$ is defined by a Shimura subdatum with underlying group $\mathrm{MT}(S)$. If $\mathrm{MT}(S) \not= {GSp}_{2g}$, then the smallest special subvariety of $\sA_g$ is not $\sA_g$, which contradicts the assumption of (i).

We have (iv) $\Rightarrow$ (i) since every special subvariety of $\sA_g$ is bi-algebraic.

The implication (v) $\Rightarrow$ (iv) is easy.

It remains to prove (ii) $\Rightarrow$ (v). Let $s \in S(\C)$ be such that $\mathrm{MT}(s) = {GSp}_{2g}$. Take $\tilde{s} \in {u}^{-1}(s)$. Let $F$ be a bi-algebraic subvariety of $\sA_g$ which contains $s$ with $\dim F > 0$. Let $\tilde{F}$ be an irreducible component of ${u}^{-1}(F)$ which contains $\tilde{s}$. It suffices to prove $\tilde{F} = \frak{H}_g$.

The proof goes as follows. By a result of Ullmo-Yafaev \cite[Theorem~1.2]{UllmoA-characterisat}, bi-algebraic subsets of $\frak{H}_g$ are precisely the weakly special subsets of $\frak{H}_g$. Hence $\tilde{F}$ is a weakly special subset of $\frak{H}_g$. By definition of weakly special subvarieties (see \cite[Definition~2.1]{UllmoA-characterisat} or \cite[Definition~4.1.(b)]{PinkA-Combination-o}), there exist a connected Shimura subdatum $(G,\sX)$ of $({GSp}_{2g},\frak{H}_g)$ and a decomposition $(G^{\mathrm{ad}},\sX) = (G_1,\sX_1) \times (G_2,\sX_2)$ and a point $\tilde{x}_2 \in \sX_2$ such that $\tilde{F} = \sX_1 \times\{\tilde{x}_2\}$. The condition $\mathrm{MT}(\tilde{s}) = {GSp}_{2g}$ implies that the smallest Shimura subdatum of $({GSp}_{2g},\frak{H}_g)$ whose underlying space contains $\tilde{s}$ is $({GSp}_{2g},\frak{H}_g)$. Therefore $(G,\sX) = ({GSp}_{2g},\frak{H}_g)$. But then $G^{\mathrm{ad}} = {GSp}_{2g}^{\mathrm{ad}}$ is a simple group, and hence either $\tilde{F} = \frak{H}_g$ or $\tilde{F}$ is a point. But $\dim\tilde{F} > 0$, so $\tilde{F}= \frak{H}_g$.
\end{proof}

\subsection{Proof of Theorem~\ref{ThmGeomCriterionHodgeGeneric}}
We may replace $S$ by $\mu_A(S)$ and hence assume that $S$ is an irreducible subvariety of $\sA_g$ of dimension $\ge g$. Recall the uniformization ${u} {:}\; \frak{H}_g \rightarrow \sA_g$ and our convention that $\tilde{S}$ is a complex analytic irreducible component of $u^{-1}(S)$.


The key to prove Theorem~\ref{ThmGeomCriterionHodgeGeneric} is the following proposition, whose proof uses Ax-Schanuel.
\begin{prop}\label{PropProperBiAlgSubset}
Suppose Condition~ACZ is satisfied. Then for any $\tilde{s} \in \tilde{S}$, there exists a bi-algebraic subset $\tilde{F}$ of positive dimension, \textbf{properly} contained in $\frak{H}_g$, such that $\tilde{s} \in \tilde{F}$.
\end{prop}
\begin{proof} 
 
Fix a $\mathbf{c} \in \C^g$ and define the following subspace of $\frak{H}_g$
\[
H_{\mathbf{c},\tilde{s}} := \{Z \in \frak{H}_g: Z \mathbf{c} = \tilde{s}  \mathbf{c}\}.
\]
Then $H_{\mathbf{c},\tilde{s}}$ has codimension $g$ in $\frak{H}_g$. Apply Condition~ACZ to this $\tilde{s} \in \tilde{S}$ and $\mathbf{c} \in \C^g$. Hence we obtain a complex analytic variety $\tilde{C}$ of dimension $\dim S - g + 1$ passing through $\tilde{s}$ such that $\tilde{C} \subset \tilde{S} \cap H_{\mathbf{c},\tilde{s}}$.

Now $\tilde{C} \subset H_{\mathbf{c},\tilde{s}}$, so we have
\begin{equation}\label{EqDimTildeCZar}
\dim \tilde{C}^{{Zar}} \le \dim H_{\mathbf{c},\tilde{s}} = \dim \frak{H}_g -g.
\end{equation}
On the other hand $\tilde{C} \subset \tilde{S}$, so ${u}(\tilde{C}) \subset {u}(\tilde{S}) = S$. Hence
\begin{equation}\label{EqDimCZar}
\dim {u}(\tilde{C})^{{Zar}} \le \dim S.
\end{equation}
 Apply Ax-Schanuel, namely Theorem~\ref{ThmAS}, to $\tilde{C}$. We obtain
\begin{equation}\label{EqDimFromAS}
\dim \tilde{C}^{\mathrm{Zar}} + \dim {u}(\tilde{C})^{{Zar}}  \ge \dim \tilde{C} + \dim \tilde{C}^{biZar}.
\end{equation}
 
Assume $\tilde{C}^{biZar} = \frak{H}_g$. Then we have
\[
(\dim \frak{H}_g - g) + \dim S  \ge \dim \tilde{C} + \dim \frak{H}_g
\]
by \eqref{EqDimTildeCZar}, \eqref{EqDimCZar} and \eqref{EqDimFromAS}. But this cannot hold since $\dim\tilde{C} = \dim S - g + 1 > 0$. Hence $\tilde{C}^{biZar} \not= \frak{H}_g$.

On the other hand $\dim \tilde{C}^{biZar} > 0$ since $\dim \tilde{C} = \dim S - g + 1 \ge 1$. So we can take the desired $\tilde{F}$ to be $\tilde{C}^{biZar}$.
\end{proof}

Before moving on, we point out that we have not yet used the full strength of Condition~ACZ since we did not vary the variable $\mathbf{c}$.

Now let us proof Theorem~\ref{ThmGeomCriterionHodgeGeneric}.
\begin{proof}[Proof of Theorem~\ref{ThmGeomCriterionHodgeGeneric}]
 
Suppose Theorem~\ref{ThmGeomCriterionHodgeGeneric} is not true.

By condition (v) of Lemma~\ref{LemmaEquivalentConditionsForProperSpecial}, there exists a point $s \in S(\C)$ such that $s$ is not contained in any proper bi-algebraic subvariety of $\sA_g$ of positive dimension.

Take $\tilde{s}$ to be a point in ${u}^{-1}(s)$ for this $s$. Applying Proposition~\ref{PropProperBiAlgSubset} to $\tilde{s}$, we get a bi-algebraic subset $\tilde{F}$ of positive dimension, properly contained in $\frak{H}_g$, such that $\tilde{s} \in \tilde{F}$. But then ${u}(\tilde{F})$ is a proper bi-algebraic subvariety of $\sA_g$ of positive dimension which contains $s$.
 Now we get a contradition.
 \end{proof}

\subsection{Proof of Theorem~\ref{ThmStrongerForm}}
In fact the same techniques for Theorem~\ref{ThmGeomCriterionHodgeGeneric} can be used to prove Theorem~\ref{ThmStrongerForm}.
\begin{proof}[Proof of Theorem~\ref{ThmStrongerForm}]
Suppose we have an abelian scheme $A/S$ satisfying the three properties.

Let $\mathbf{c} \in \C^g$ and $\tilde{s} \in \tilde{S}$ be as in condition (ii) of Theorem~\ref{ThmStrongerForm}. Then for $H_{\mathbf{c},\tilde{s}} = \{Z \in \frak{H}_g : Z \mathbf{c} = \tilde{s} \mathbf{c} \} \subset \frak{H}_g$, we have
\begin{equation}\label{EqCodimensionHypersurfaceInBizar}
{\rm{codim}}_{\tilde{S}^{\mathrm{biZar}}}(H_{\mathbf{c},\tilde{s}} \cap \tilde{S}^{\mathrm{biZar}}) = g.
\end{equation}
Now that $\tilde{S}^{\mathrm{biZar}}$ is affine linear in $\frak{H}_g$ by Lemma~\ref{LemmaWeaklySpecialLinear}. 
So for such a $\mathbf{c}$, \eqref{EqCodimensionHypersurfaceInBizar} holds for any $\tilde{s} \in \tilde{S}$ because $H_{\mathbf{c},\tilde{s}}$ is also affine linear in $\frak{H}_g$. Hence we may assume that $\tilde{s}$ is Hodge generic in $\tilde{S}$, namely $\mathrm{MT}(\tilde{s}) = \mathrm{MT}(S)$. 

Applying Condition~ACZ to this $\tilde{s}$ and $\mathbf{c}$, we obtain a complex analytic variety $\tilde{C}$ of dimension $\dim S - g + 1$ passing through $\tilde{s}$ such that $\tilde{C} \subset \tilde{S} \cap H_{\mathbf{c},\tilde{s}}$. Then \eqref{EqCodimensionHypersurfaceInBizar} implies
\begin{equation}\label{EqDimTildeCZarGeneral}
\dim \tilde{C}^{{Zar}} \le \dim (H_{\mathbf{c},\tilde{s}} \cap \tilde{S}^{\mathrm{biZar}}) = \dim \tilde{S}^{\mathrm{biZar}} - g.
\end{equation}
On the other hand $\tilde{C} \subset \tilde{S}$, so ${u}(\tilde{C}) \subset {u}(\tilde{S}) = S$. Hence
\begin{equation}\label{EqDimCZarGeneral}
\dim {u}(\tilde{C})^{{Zar}} \le \dim S.
\end{equation}
Apply Ax-Schanuel, namely Theorem~\ref{ThmAS}, to $\tilde{C}$. We obtain
\begin{equation}\label{EqDimFromASGeneral}
\dim \tilde{C}^{{Zar}} + \dim {u}(\tilde{C})^{{Zar}} \ge \dim \tilde{C} + \dim \tilde{C}^{biZar} = \dim S - g + 1 + \dim \tilde{C}^{\mathrm{biZar}}.
\end{equation}

By \eqref{EqDimTildeCZarGeneral}, \eqref{EqDimCZarGeneral} and \eqref{EqDimFromASGeneral}, we get
\[
\dim \tilde{S}^{\mathrm{biZar}} > \dim \tilde{C}^{\mathrm{biZar}}.
\]
Thus in order to get a contradiction, it suffice to prove $\tilde{S}^{\mathrm{biZar}} = \tilde{C}^{\mathrm{biZar}}$. We shall use condition (i) of Theorem~\ref{ThmStrongerForm} to prove this fact.

The logarithmic Ax theorem for $\sA_g$ says that $\tilde{S}^{\mathrm{biZar}} = H_S^{\circ}(\R)^+\tilde{s}$. We refer to \cite[Theorem~8.1]{GaoTowards-the-And} for this theorem. Recall that $\tilde{s}$ is Hodge generic in $\tilde{S}$. Hence $H_S^{\circ}$ is normal in $\mathrm{MT}(\tilde{s})$ by Andr\'{e} \cite[Theorem~1]{A1}.

By Ullmo-Yafaev \cite[Theorem~1.2]{UllmoA-characterisat}, bi-algebraic subsets of $\frak{H}_g$ are precisely the weakly special subsets of $\frak{H}_g$. Now $\tilde{C}^{\mathrm{biZar}}$ contains $\tilde{s}$ which is Hodge generic in $\tilde{S}$, and $\tilde{C}^{\mathrm{biZar}} \subset \tilde{S}^{\mathrm{biZar}} \subset \mathrm{MT}(\tilde{s})(\R)^+\tilde{s}$. So by definition of weakly special subvarieties (see \cite[Definition~2.1]{UllmoA-characterisat} or \cite[Definition~4.1.(b)]{PinkA-Combination-o}), we have $\tilde{C}^{\mathrm{biZar}} = N(\R)^+\tilde{s}$ for some normal subgroup $N$ of $\mathrm{MT}(\tilde{s})$.

Now $N(\R)^+\tilde{s} \subset H_S^{\circ}(\R)^+\tilde{s}$, both $N$ and $H_S^{\circ}$ are normal subgroups of the reductive group $\mathrm{MT}(\tilde{s})$, and $H_S^{\circ}$ is simple by condition (i) of Theorem~\ref{ThmStrongerForm}. Hence $N(\R)^+\tilde{s} = H_S^{\circ}(\R)^+\tilde{s}$, and so $\tilde{C}^{\mathrm{biZar}} = \tilde{S}^{\mathrm{biZar}}$.
\end{proof}

 \end{sloppypar}

 \end{document}